\documentclass[final,leqno]{article}

\usepackage{algorithm}
\usepackage{algorithmic}
\usepackage{hyperref}

\usepackage[normalem]{ulem}
\usepackage{cancel}
\usepackage{pgf-filehook}

\usepackage{amsfonts,amssymb,amsmath,cmmib57,color,graphics}
\usepackage{amstext,amsthm}
\usepackage{amscd}
\usepackage{graphicx,url, pgf}
\usepackage{comment}
\usepackage{longtable}
\usepackage{calc}
\usepackage[utf8]{inputenc}
\usepackage{tikz}

\newtheorem{prop}{Proposition}

\newtheorem{lemma}[prop]{Lemma}
\newtheorem{remark}{Remark}
\newtheorem{theorem}[prop]{Theorem}
\newtheorem{corollary}[prop]{Corollary}

\newcommand{\wt}{\widetilde}

\newcommand{\iunit}{{\mathfrak i}}

\newcommand{\la}{\lambda}

\newcommand{\R}{\mathbb R}
\newcommand{\C}{\mathbb C}

\renewcommand{\theta}{\vartheta}

\usepackage{todonotes}

\newcommand{\W}{\mathcal W}
\newcommand{\w}{_{_{\mathcal W}}}
\newcommand{\Z}{\mathbb Z}

\newcommand{\Holder}{H\"older continuous}
\newcommand{\QT}{\mathcal Q\mathcal T}
\newcommand{\K}{\mathcal K}
\renewcommand{\epsilon}{\varepsilon}
\newcommand{\spess}{{\rm sp_{ess}}}

\newcommand{\rv}{real valued}
\newcommand{\QTm}{$\QT$ matrix}
\newcommand{\QTmm}{$\QT$ matrices}

\begin{document}

\title{Geometric means of quasi-Toeplitz matrices\thanks{The work of the first two authors was partly supported by INdAM (Istituto Nazionale di Alta Matematica) through a GNCS Project. A part of this research has
been done during a visit of the third author to the University of Perugia. Version of \today}}
\author{
Dario A. Bini\thanks{Dipartimento di Matematica, Università di Pisa, Italy ({\tt dario.bini@unipi.it, jie.meng@dm.unipi.it})} 
\and 
Bruno Iannazzo\thanks{Dipartimento di Matematica e Informatica, Universit\`a degli Studi di Perugia, Italy ({\tt bruno.iannazzo@unipg.it})}
\and 
$\mbox{Jie Meng}^{\dagger}$
}
 \date{}\thispagestyle{empty}
 \maketitle

\begin{abstract}
We study means of geometric type of quasi-Toeplitz matrices, that are semi-infinite matrices $A=(a_{i,j})_{i,j=1,2,\ldots}$ of the form $A=T(a)+E$, where $E$ represents a compact operator, and $T(a)$ is a semi-infinite Toeplitz matrix associated with the function $a$, with Fourier series $\sum_{\ell=-\infty}^{\infty} a_\ell e^{\iunit \ell t}$, in the sense that  $(T(a))_{i,j}=a_{j-i}$. If $a$ is \rv\ and essentially bounded, then these matrices represent bounded self-adjoint operators on~$\ell^2$. We consider the case where $a$ is a continuous function, where quasi-Toeplitz matrices coincide with a classical Toeplitz algebra, and the case where $a$ is in the Wiener algebra, that is, has absolutely convergent Fourier series.

We prove that if $a_1,\ldots,a_p$ are continuous and positive functions, or are in the Wiener algebra with some further conditions, then means of geometric type, such as the ALM, the NBMP and the Karcher mean of quasi-Toeplitz positive definite matrices associated with $a_1,\ldots,a_p$, are quasi-Toeplitz matrices associated with the geometric mean $(a_1\cdots a_p)^{1/p}$, which differ only by the compact correction. We show by numerical tests that these operator means can be practically approximated.

\bigskip

\end{abstract}

\section{Introduction}

The concept of matrix geometric mean
of a set of positive definite matrices, together with its analysis and computation, has received an increasing interest in the past years due to its rich and elegant theoretical properties,
the nontrivial algorithmic issues related to its computation, and also the important role played in several applications. 

A fruitful approach to get a matrix geometric mean has been to identify and impose the right properties required by this function. These properties, known as Ando-Li-Mathias axioms, are listed in the seminal work \cite{alm}.
It is interesting to point out that, while in the case of two positive definite matrices $A_1,A_2$ the geometric mean $G=G(A_1,A_2)$ is uniquely defined, in the case of $p>2$ matrices $A_1,\ldots,A_p$ there are infinitely many functions fulfilling the Ando-Li-Mathias axioms. These axioms are satisfied, in particular, by the ALM mean \cite{alm}, the NBMP mean, independently introduced in \cite{nakamura} and \cite{bmp}, and by the Riemannian center of mass, the so-called Karcher mean, identified as a matrix geometric mean in \cite{moakher-2005}. We refer the reader to the book \cite{bhatia:book} for an introduction to matrix geometric means with historical remarks.

These means have been generalized in a natural way to the infinite dimensional case. For instance, the Karcher mean
has been extended to self-adjoint positive definite operators and to self-adjoint elements of a  $C^*$-algebras in \cite{lawson-lim-2014} and \cite{lawson-2020}, respectively. Another natural generalization relies on the concept of weighted mean: for instance, a generalization of the NBMP mean is presented in \cite{lll-12}. 

Matrix geometric means have played an important role in several applications such as radar detection \cite{barbaresco-2008}, \cite{barbaresco-2010}, image processing \cite{rathi}, elasticity \cite{moakher-2006};  while more recent applications include machine learning \cite{appl:ml1}, \cite{appl:ml2}, brain-computer interface \cite{appl:bci1}, \cite{appl:bci2}, and network analysis \cite{fi}. 
The demand from applications has required some effort in the design and analysis of numerical methods for matrix geometric means.
See in particular \cite{bmp}, \cite{bini-iannazzo-2013},   \cite{bini-iannazzo-2011}, \cite{iannazzo-2016}, 
\cite{jeuris-et-al-2012}, 
 \cite{iporcelli17}, \cite{huang}, and the literature cited therein, where algorithmic issues have been investigated. 

In certain applications, 
the matrices to be averaged have further structures originated by the peculiar features of the physical model that they describe. In particular, in the radar detection application, the shift invariance properties of some quantities involved in the model turn into the Toeplitz structure of the matrices. A (possibly infinite) matrix $T=(t_{i,j})$ is said to be Toeplitz if $ t_{i,j}=a_{j-i}$ for some $\{a_k\}_{k\in\mathbb Z}\subset\mathbb C$. The sequence $\{a_k\}_{k\in \mathbb Z}$ may arise as the set of Fourier coefficients of a function $a$ defined on the unit circle $\mathbb T=\{z\in\mathbb C\,:\,|z|=1\}$, with Fourier series $\sum_{\ell=-\infty}^\infty a_\ell z^\ell$. 
The function $a$ is said to be the symbol associated with the Toeplitz matrix, and the Toeplitz matrix is denoted by $T(a)$.

The problem of averaging finite-dimensional Toeplitz matrices has been treated in some recent papers, see for instance \cite{bini-et-al-2014}, \cite{jeuris-vandebril}, \cite{nobari-2016}, \cite{nobari-2018}, with the aim to provide a definition of matrix geometric mean $G=G(A_1,\ldots,A_p)$ 
which preserves the matrix structure of the input matrices $A_k$, for $k=1,\ldots,p$.

In this paper, we consider the problem of analyzing the structure of the matrix $G$ when $A_k=T(a_k)$, for $k=1,\ldots,p$, is a self-adjoint positive definite Toeplitz matrix and $G$ is either the ALM, NBMP, and their weighted counter-parts, or the Karcher mean. We start by considering the case where the symbols $a_k(z)$, $k=1,\ldots,p$, are continuous \rv\ positive functions. 

We show that if $G$ is the ALM, NBMP, or a weighted mean of the self-adjoint positive definite operators $T(a_1),\ldots,T(a_p)$, then, even though the Toeplitz structure is lost in $G$, a hidden structure is preserved. In fact, we show that $G$ is a quasi-Toeplitz matrix, that is, it can be written as $G=T(g)+K_G$, where $K_G$ is a self-adjoint compact  operator on the set $\ell^2$ formed by infinite vectors $v=(v_i)_{i\in\mathbb Z^+}$, $v_i\in\mathbb C$ such that $\sum_{i=1}^\infty |v_i|^2<\infty$. An interesting feature is that the function $g(z)$ turns out to be $(a_1(z)a_2(z)\cdots a_p(z))^{\frac{1}{p}}$, as expected. These results are extended to the more general case $A_k=T(a_k)+K_k$, where $K_k$ is a compact operator.

Among the side results that we have introduced to prove the main properties of the geometric mean, it is interesting to point out Theorem \ref{thm:f(A)} which states that for any continuous \rv\ function $a(z)$ and for any continuous function $f(t)$ defined on the spectrum of $T(a)$, the difference $f(T(a))-T(f(a))$ is a compact operator. 

Algebras of quasi-Toeplitz matrices are also known in the literature as \emph{Toeplitz algebras}.
See for instance \cite{BS:book99} where in the Example 2.8 it is considered the case of continuous symbols, or in Section 7.3, where the smallest closed subalgebra of $\mathcal B(\ell^p)$, $1<p<\infty$, containing the matrices of the type $T(a)$, where $a$ has an absolute convergent Fourier series, is considered. Here, $\mathcal B(\ell^p)$ denotes the set of bounded operators on the set $\ell^p$ of infinite vectors $v=(v_i)_{i\in\mathbb Z^+}$ such that $\sum_{i=1}^\infty |v_i|^p$ is finite.

After analyzing the case of continuous symbols, we identify additional regularity assumptions on the functions $a_1(z),\ldots,a_p(z)$, such that the summation
$\sum_{\ell=-\infty}^{+\infty}|g_\ell|$ is finite, where $\sum_{\ell=-\infty}^{+\infty}g_\ell z^\ell$ is the Fourier series of $g(z)=(a_1(z)\cdots a_p(z))^{1/p}$. This property is meaningful since it implies that $T(a_k)$ as well as $G$ are bounded operators on the set $\ell^\infty$ of infinite vectors with uniformly bounded components.

We will apply these ideas also in the case where the matrices $A_1,\ldots,A_p$, are of finite size $n$ and show that, for a sufficiently large value of $n$, the matrix $G$ is numerically well approximated by a Toeplitz matrix plus a matrix correction that is nonzero in the entries close to the top left corner and to the bottom right corner of the matrix.

Finally we discuss some computational issues. Namely, we examine classical algorithms for computing the square root and the $p$th root of a matrix as Cyclic Reduction and Newton's iteration, and analyze their convergence properties when applied to Toeplitz matrices given in terms of the associated symbols. Then we deal with the problem of computing $G$. We introduce and analyze algorithms for computing the Toeplitz part $T(g)$ and the correction part $K_G$ in a very efficient way, both in the case of infinite and of finite Toeplitz matrices. 
A set of numerical experiments shows that the geometric means of infinite quasi-Toeplitz matrices can be easily and accurately computed by relying on simple MATLAB implementations, based on the CQT-Toolbox of \cite{bmr}.

The paper is organized as follows: in the next section, after providing some preliminary results on Banach algebras and $C^*$-algebras, we
recall the most common matrix geometric means, introduce the class of quasi-Toeplitz matrices associated with a continuous symbol and study functions of matrices in this class. Moreover, we analyze the case where the symbol of quasi-Toeplitz matrices is in the Wiener class, that is, the associated Fourier series is absolutely convergent. In Section \ref{sec:planar} we prove that the most common definitions of geometric mean of operators preserve the structure of quasi-Toeplitz matrix, and that the symbol associated with the geometric mean is the geometric mean of the symbols associated with the input matrices. In the same section, we give regularity conditions on the symbols in order that the convergence of the ALM sequence of the symbol holds in the Wiener norm, this implies the convergence of the operator sequences in the infinity norm. In Section \ref{sec:issues} we discuss some issues related to the computation of the $p$th root and of the geometric mean of quasi-Toeplitz matrices, in Section \ref{sec:finite} we describe the geometric means  of finite quasi-Toeplitz matrices, while in Section \ref{sec:exp} we report the results of some numerical experiments. Finally, Section \ref{sec:conc} draws some conclusions and open problems.

\section{Preliminaries}
\subsection{Notation} Let $\iunit$ be the complex unit,
and denote $\mathbb T=\{z\in\mathbb C\,:\, |z|=1\}$ the unit circle in the complex plane. 
Given a function $f(z):\mathbb T\to\mathbb C$ the composition $\wt f(t)=f(e^{\iunit t})$ is a $2\pi$-periodic function from $\mathbb R$ to $\mathbb C$ which can be restricted to $[0,2\pi]$. For the sake of simplicity, we keep the notation $f(t)$ to denote $\wt f(t)$, where it should be understood that the $t$ variable is real and ranges in the set $[0,2\pi]$, while the $z$ variable is complex and ranges in the set $\mathbb T$.

For a given subset $\mathcal S\subset \mathbb R$, and $1\le p<\infty$, 
we denote by $L^p(\mathcal S)$ the set of 
functions $f(t):\mathcal S\to\mathbb C$, such that 
$\|f\|_{p,\mathcal S}=(\int_\mathcal S |f(t)|^p)^\frac 1p<\infty$.
We denote $L^\infty(\mathcal S)$ the set of measurable functions defined over $\mathcal S$
with finite essential supremum and we set
$\|f\|_{\infty,\mathcal S}=\mbox{ess}\sup_{t\in\mathcal S}|f(t)|$.
If the set $\mathcal S$ is clear from the context we write $\|f\|_{L^p}$ in place of $\|f\|_{p,\mathcal S}$ for $1\leq p<\infty$ and $\|f\|_{\infty}$ in place of $\|f\|_{\infty,\mathcal S}$.
 
We denote by $C(\mathcal S)$ the set of continuous functions $f:\mathcal S\to \C$.
The subset of $f\in C([0,2\pi])$ such that $f(0)=f(2\pi)$ is denoted by $C_{2\pi}$, while 
  $L^p_{2\pi}
=L^p([0,2\pi))$.

Let $\mathbb Z^+$ be the set of positive integers and denote by $\ell^p$, with $1\le p<\infty$, the set of infinite sequences $v=(v_i)_{i\in\mathbb Z^+}$, $v_i\in\mathbb C$,  such that $\|v\|_p:=(\sum_{i=1}^\infty |v_i|^p)^\frac{1}{p}$ is finite, while $\ell^\infty$ is the set of sequences such that $\|v\|_\infty=\sup_i|v_i|<\infty$.
If $A=(a_{i,j})_{i,j=1,\ldots,n}\in\C^{n\times n}$ is such that $a_{i,j}\in\mathbb C$, then the function $\|A\|_2=\max_{\|x\|=1}\|Ax\|_2$ is the matrix norm induced by the vector norm $\|x\|_2=(\sum_{i=1}^n|x_i|^2)^\frac{1}{2}$.
Similarly, if $A:\ell^p\to\ell^p$ is a linear  operator, the function $\|A\|_p=\sup_{\|v\|_p=1}\|Av\|_p$, 
is the operator norm induced by the $\ell^p$ norm $\|\cdot\|_p$.

Recall that any linear operator $A:\ell^2\to\ell^2$ can be represented by a semi-infinite matrix $(a_{i,j})_{i,j\in\mathbb Z^+}$, which we denote with the same symbol $A$. The set of bounded operators onto $\ell^2$ is denoted by $\mathcal B(\ell^2)$, or more simply by $\mathcal B$.
We recall that the set $\mathcal K\subset\mathcal B$ of compact operators is the closure in $\mathcal B$ of the set of bounded operators with finite rank. 

Given the matrix (operator) $A=(a_{i,j})$, we define $A^*=(\overline a_{j,i})$ the conjugate transpose of $A$, where $\overline x$ denotes the complex conjugate of $x\in\mathbb C$. A matrix (operator) $A$ is Hermitian or self-adjoint if $A=A^*$, moreover, is 
positive 
if in addition $v^*Av=\sum_i \overline v_i(Av)_i>0$ for any $v\in\mathbb C^n\setminus\{0\}$ ($v\in\ell^2\setminus\{0\}$). 
We say that $A$ is positive definite if there exists a constant $\gamma>0$ such that 
$v^*Av\geq \gamma \|v\|_2^2$ for any $v\in \ell^2$, $v\ne 0$.

\subsection{Banach algebras and $C^*$-algebras} A Banach algebra $\mathfrak B$ is a Banach space with a product such that $\|ab\|\le \|a\|\|b\|$, for $a,b\in\mathfrak B$, where $\|\cdot\|$ is the norm of $\mathfrak B$. Trivial examples of Banach algebras that we will use are: the set of essentially bounded functions $L^\infty(S)$, with $S\subset \R$,  with the norm $\|f\|_{\infty,S}$ 
and pointwise multiplication (as a special case we have $L_{2\pi}^\infty$); the set $C(K)$ of continuous functions on a compact $K\subset \C^n$ with pointwise multiplication; and $\mathcal B(\ell^2)$ with operator composition. 
If $a^{(k)}\in\mathfrak B$ is a sequence such that $\lim_k\|a^{(k)}-a\|=0$ for some $a\in\mathfrak B$, we write $a^{(k)}\to_{\mathfrak B} a$. 

An immediate consequence of the definition of Banach algebra is the following.
\begin{lemma}\label{thm:ba1}
Let $\mathfrak B$ be a Banach algebra and $a_1,\ldots,a_p\in\mathfrak B$. If the sequences $\{a_i^{(k)}\}_k\subset \mathfrak B$   are such that 
$a_i^{(k)}\to_{\mathfrak B} a_i$ for $i=1,\ldots,p$, then
\[
	a_1^{(k)}\cdots a_p^{(k)}\to_{\mathfrak B} a_1\cdots a_p.
\]
\end{lemma}
\begin{proof}
Let $\|\cdot\|$ be the norm in $\mathfrak B$. We proceed by induction on $p$. The case $p=1$ is trivial and if the property is true for $p-1$ sequences, then
\[\begin{aligned}
\|a_1^{(k)}\cdots\ & a_p^{(k)}-a_1\cdots a_p\|\\
&\le \|a_1^{(k)}\cdots a_p^{(k)}-a_1^{(k)}a_2\cdots 
a_p\|+\|a_1^{(k)}a_2\cdots a_p-a_1\cdots a_p\|\\
&\le \|a_1^{(k)}\|\|a_2^{(k)}\cdots a_p^{(k)}-a_2\cdots a_p\|+\|a_1^{(k)}-a_1\|\|a_2\cdots a_p\|,
\end{aligned}
\]
and the latter tends to zero since $\|a_2^{(k)}\cdots a_p^{(k)}-a_2\cdots a_p\|$ tends to zero by inductive hypothesis, $\|a_1^{(k)}-a_1\|$ tends to zero by hypothesis, $\{\|a_1^{(k)}\|\}$ is uniformly bounded since it converges and $\|a_2\cdots a_p\|\le \|a_2\|\cdots \|a_p\|$ that is bounded.
\end{proof}
To any element $A$ of a Banach algebra, it is assigned the spectrum $\mathrm{sp}(A)$ that is the set of $\la\in\C$ such that $A-\la I$ fails to be invertible. The spectral radius $\rho(A)$ is defined as $\rho(A)=\sup_{\lambda\in\mathrm{sp}(A)} |\lambda|$.

A $C^*$-algebra is a Banach algebra $\mathfrak B$ with an involution $A\to A^*$ such that $(aS+bT)^*=\overline a S^*+\overline b T^*$; $(ST)^*=T^*S^*$, $(T^*)^*=T$ and $\|T^*T\|=\|T^2\|$, for $T,S\in\mathfrak B$ and $a,b\in \mathbb C$. Two important examples of $C^*$-algebras are $\mathcal B(\ell^2)$ with the adjoint operation and $C(K)$, for $K\subset \mathbb R$ compact, with the pointwise conjugation. 

A $C^*$-subalgebra $\mathfrak F\subset\mathfrak B$ is a closed subalgebra that contains the identity and such that $A\in\mathfrak F$ implies $A^*\in\mathfrak F$.

An element $A$ of a $C^*$-algebra is self-adjoint if $A^*=A$, normal if $A^*A=AA^*$. We need some results about $C^*$-algebras taken from Proposition 4.1.1, and Theorems 4.1.3 and 4.1.6 of \cite{kadison}.

\begin{lemma}\label{thm:lemC1}
Let $A$ be an element of a $C^*$-algebra $\mathfrak B$.  
\begin{enumerate}
\item[{\rm (i)}] If $A$ is normal then $\rho(A)=\|A\|$. Moreover, if $A$ is self-adjoint, then the spectrum $\mathrm{sp}(A)$ of $A$ is a compact subset of $\R$. 	

\item[{\rm (ii)}]
If  $A$ is self-adjoint, then there exists a unique continuous mapping $f\to f(A)\,:C(\mathrm{sp}(A))\to\mathfrak B$ such that $f(A)$ has its elementary meaning when $f$ is a polynomial. Moreover,  $f(A)$ is normal, while it is self-adjoint if and only if $f$ takes real values at the spectrum of $A$. 

\item[{\rm (iii)}]
If $A$ is self-adjoint and $f\in C(\mathrm{sp}(A))$, then $\|f(A)\|=\|f\|_{\infty,\mathrm{sp}(A)}$ and $\mathrm{sp}(f(A))=f(\mathrm{sp}(A))$.
\end{enumerate}
\end{lemma}

Lemma \ref{thm:lemC1} implies the existence of a \emph{continuous functional calculus}, that is to define a function of a self-adjoint element $A$ of a $C^*$-algebra when the function is continuous at the spectrum of $A$.

\subsection{Means of operators}

A great effort has been done to define properly the geometric mean of two or more positive definite matrices. The weighted geometric mean of two $n\times n$ matrices $A$ and $B$, with weight $t\in[0,1]$, is defined as
\[
	A\#_t B=
	A^{1/2}(A^{-1/2}BA^{-1/2})^tA^{1/2},
\]
while the geometric mean is $G(A,B)=A\# B=A\#_{1/2}B$. 
Notice that, using continuous functional calculus, the same formula makes sense also for bounded self-adjoint positive definite operators $A,B\in\mathcal B(\ell^2)$.

For more than two matrices, several attempts have been done before the right definition were fully understood. For the ease of the reader, here we limit the treatise to three positive matrices of the same size, while we refer to a later section for the case of more than three matrices.
The ALM mean, proposed by Ando, Li and Mathias \cite{alm}, has been introduced as the common limit of the sequences
\[
	A_{k+1}=B_k\# C_k,\quad
	B_{k+1}=C_k\# A_k,\quad
	C_{k+1}=A_k\# B_k,
\]
with $A_0=A, B_0=B, C_0=C$.
A nice feature of this mean is that the convergence of the sequence can be obtained using the Thompson metric \cite{thompson}
\[
	d(A,B)=\log \max\{\rho(A^{-1}B),\rho(B^{-1}A)\}. 
\]
It is well known that convergence in the Thompson metric implies convergence in the operator norm induced by the $\ell^2$ norm, also in the infinite dimensional case. This fact allows one to define the ALM mean of self-adjoint positive definite operators.

A variant of this construction, introduced independently by Bini, Meini, Poloni \cite{bmp} and Nakamura \cite{nakamura}, named NBMP mean is obtained by the sequences
\[
	A_{k+1}=A_k\#_{2/3}(B_k\# C_k),~\,
	B_{k+1}=B_k\#_{2/3}(C_k\# A_k),~\,
	C_{k+1}=C_k\#_{2/3}(A_k\# B_k),
\]
with $A_0=A$, $B_0=B$, $C_0=C$. The sequences converge to a common limit, different in general from the ALM mean, and with a faster rate than the ALM sequence.
The convergence of these sequences in the Thompson metric allows one to extend them to operators.
Weighted versions of the ALM and NBMP mean can be given by introducing suitable parameters, see Section \ref{sec:planar} for more details. 

A well-recognized geometric mean of matrices is the unique positive definite solution of the matrix equation
\[
	\log(X^{-1/2}AX^{-1/2})+\log(X^{-1/2}BX^{-1/2})+\log(X^{-1/2}CX^{-1/2})=0,
\]
the so-called Karcher mean of $A,B$ and $C$.
Its extension to operators has been more complicated, but it has been shown that the Karcher mean can be defined also for positive definite self-adjoint bounded operators, proving that the equation above has a unique positive definite self-adjoint solution \cite{lawson-lim-2014}. 

\subsection{Quasi-Toeplitz matrices with continuous symbols}

An integrable and $2\pi$-periodic function $a\in L^1_{2\pi}$, with Fourier series $\sum_{k=-\infty}^{\infty} a_ke^{\iunit k t}$ can be associated with the semi-infinite Toeplitz matrix $T(a)=(t_{i,j})_{i,j\in\mathbb Z^+}$ such that  $t_{i,j}=a_{j-i}$. The function $a$ is said to be the symbol associated with the Toeplitz matrix $T(a)$. A classical result states that $T(a)$ represents a bounded operator on $\ell^2$ if and only if $a\in L^{\infty}_{2\pi}$  \cite[Theorem 1.1]{BG:book00}. Moreover, if $a$ is continuous, then $\|T(a)\|_2=\|a\|_\infty$.

Toeplitz matrices form a linear subspace of $\mathcal B(\ell^2)$, closed by involution since $T(a)^*=T(\overline a)$, but not closed under multiplication, that is, by composition of operators in $B(\ell^2)$, so they are not a $C^*$-subalgebra of $\mathcal B(\ell^2)$. Nevertheless, there is a nice formula for the product of two Toeplitz matrices  \cite[Propositions 1.10 and 1.11]{BG:book00}.
\begin{theorem}\label{th:widom}
Let $a,b\in L_{2\pi}^{\infty}$,
\begin{equation}\label{eq:Tab}
    T(a)T(b)=T(ab)-H(a^-)H(b^+),
\end{equation}
where  $H(a^-)_{i,j}=(a_{-i-j+1})_{i,j\in\mathbb Z^+}$ and $H(b^+)=(b_{i+j-1})_{i,j\in\mathbb Z^+}$. Moreover, if $a,b$ are continuous functions then 
 $H(a^-),H(b^+)$ are compact operators on $\ell^2$.
\end{theorem}

A nice feature of equation \eqref{eq:Tab} is that it relates the symbol of the product to the product of the symbols and if $a,b\in C_{2\pi}$ then $T(a)T(b)-T(ab)$ belongs to the set $\K$ of compact operators on $\ell^2$ \cite{BG:book00}.

The smallest $C^*$-subalgebra of the $C^*$-algebra $\mathcal B(\ell^2)$ containing all Toeplitz operators with continuous symbols is \cite{BS:book99} 
\[
	\QT=\{T(a)+K\,:\, a\in C_{2\pi},K \in\K\},
\]
that we call the set of \emph{quasi-Toeplitz} matrices and it is also known as Toeplitz algebra. Notice that for $A\in\QT$ there exists unique $a\in C_{2\pi}$ and $K\in\K$ such that $A=T(a)+K$, because the intersection between Toeplitz matrices with continuous symbols and compact operators is the zero operator \cite[Proposition 1.2]{BG:book00}. To define a \QTm, we will use often the notation $A=T(a)+K\in\QT$, without saying explicitly that  $a$ is a continuous and $2\pi$-periodic function and $K$ is a compact operator.

The following lemma collects and resumes some known results from \cite[Sections 1.4 and 1.5]{BG:book00}, \cite[Section 1.1]{BS:book99}  and from \cite{stefan1964}, which will be useful in the sequel. In particular, it states properties of the spectrum $\mbox{sp}(A)$ of $A$, of the essential spectrum $\spess(A)$ defined by
$\spess(A)=\{\lambda\in \mathbb C: A-\lambda I \ \hbox{is 
 not Fredholm on}\ \mathcal B(\ell^2)\}$, and on the numerical range of $A$ defined as ${\rm W}(A)=\{x^*Ax\,:\,\|x\|=1\}$.

\begin{lemma}\label{lem:toeplitz}
The following properties hold:
\begin{enumerate}
\item[{\rm (i)}] $T(a)$ represents a bounded operator on $\ell^2$ if and only if $a\in  L_{2\pi}^{\infty}$,
and $T(a)$ is self adjoint if and only if $a$ is \rv. 
\end{enumerate}
Moreover, if $a\in C_{2\pi}$ then:
\begin{enumerate}\setcounter{enumi}{1}
\item[{\rm (ii)}] $\|a\|_\infty=\|T(a)\|_2\le \|T(a)+K\|_2$, for $K\in\mathcal K$; 
\item[{\rm (iii)}] $\mbox{\rm sp}(T(a))=a([0,2\pi])\cup\{\lambda\in\mathbb C\,:\,\mbox{\rm wind}(a-\lambda)\ne 0\}$ is compact;
\item[{\rm (iv)}] $a([0,2\pi])=\spess(T(a))=\spess(T(a)+K)\subset  \mbox{\rm sp}(T(a)+K)$ for $K\in\mathcal K$;
\item[{\rm (v)}] ${\rm sp} (A)$ is contained in the closure of  ${\rm W}(A)$ for $A\in\mathcal B(\ell^2)$.
\end{enumerate}
\end{lemma}

The following result will be useful.

\begin{lemma}\label{lem:pos}
Let $A=T(a)+K\in\QT$. If $A=A^*$ then $T(a)=T(a)^*$ and $K=K^*$, moreover, $a$ is \rv. If $A$ is positive, then $a$ is nonnegative. If in addition $A$ is positive definite then $a$ is strictly positive.
\end{lemma}
\begin{proof}
If $A=A^*$, then $T(a)-T(a)^* +K-K^*=0$, and by the uniqueness of the decomposition of a quasi-Toeplitz matrix, we have that $K=K^*$ and $T(a)=T(a)^*$, that in turn implies that $a_k=\overline a_{-k}$ and thus $a(z)$ is real. 
 If $A$ is positive definite, there exists $\gamma>0$ such that  $x^*Ax\ge \gamma \|x\|_2>0$ for any $x\ne 0$. That is, the numerical range 
${\rm W}(A)$ of $A$  is formed by real numbers greater than or equal to $\gamma$ so that its closure contains positive values. 
Since by Lemma \ref{lem:toeplitz}, part (v), the spectrum ${\rm sp} (A)$ is contained in the closure of  ${\rm W}(A)$,
we have that $\lambda>0$ for any $\lambda\in {\rm sp}(A)$. From Lemma \ref{lem:toeplitz} part (iv), it follows that 
$a(t)\in {\rm sp}(A)$ so that $a(t)>0$. The case where $A$ is positive is similarly treated since $x^*Ax\ge 0$ for any $x\in\ell^2$ so that the closure of ${\rm W}(A)$ is formed by nonnegative numbers.
\end{proof}

In view of Lemma \ref{thm:lemC1}, the fact that $\QT$ is a $C^*$-algebra allows one to use continuous functional calculus. If $A$ is a self-adjoint quasi-Toeplitz matrix and $f$ is a function continuous on the spectrum of $A$, then one can define the normal quasi-Toeplitz matrix $f(A)$.
In particular, the $p$-th root of a self-adjoint positive definite quasi-Toeplitz matrix turns out to be a quasi-Toeplitz matrix.
This fact implies that the sequences generated by the ALM and the NBMP constructions, if the initial values are \QTmm, are formed by entries belonging to $\QT$.
We will prove that also the limit $G$ of these sequences belongs to $\QT$, and that it can be written as $G=T(g)+K_G$, where $g(t)$ is the geometric mean of the symbols associated with the Toeplitz part of the given matrices. Similarly, we will show that 
the Karcher mean of quasi-Toeplitz matrices is a quasi-Toeplitz matrix (the latter follows also from \cite{lawson-2020}).

In order to prove these properties we use the following results.

\begin{theorem}\label{thm:f(A)}
Let $a\in C_{2\pi}$ be a \rv\ function and $f$ a continuous function on the spectrum of $T(a)$, then
\begin{equation}\label{eq:bimformula}
	f(T(a))=T(f\circ a)+K,
\end{equation}
where $K$ is a compact operator. If $f$ takes real values on the spectrum of $T(a)$, then $K$ is self-adjoint and $\mathrm{sp}(f(T(a)))=\mathrm{sp}(T(f\circ a))$.
\end{theorem}
\begin{proof}
Since $a$ is \rv, by Lemma \ref{lem:toeplitz} the operator $T(a)$ is self-adjoint and since $a$ is continuous, its spectrum $\Sigma$ is a compact set containing the range of $a$ (see Lemma \ref{lem:toeplitz}). Moreover, one can define $f(T(a))$, for any function $f$ continuous on $\Sigma$, using the continuous functional calculus and, since the {function $f$ }is continuous in $\Sigma$, by Lemma \ref{thm:lemC1}, we have $\|f(T(a))\|_2=\|f\|_{\infty,\Sigma}$.

From Theorem \ref{th:widom} it follows that if $f$ is a polynomial, then equation \eqref{eq:bimformula} holds. By using an approximation argument, we prove that \eqref{eq:bimformula} still holds for any continuous function $f$. Indeed, $f$ can be approximated uniformly by a sequence of polynomials $p_n$ such that $\|f-p_n\|_{\infty,\Sigma}$ tends to 0 (by the Weierstrass theorem) and there exist compact operators $K_k$ such that
\[
	p_k(T(a))=T(p_k\circ a)+K_k.
\]
Since continuous functional calculus preserves $C^*$-algebras, then $f(T(a))\in \QT$, that is $f(T(a))=T(b)+K$, what we need to prove is that $b=f\circ a$.

In order to get the result, in view of the uniqueness of the decomposition of a quasi-Toeplitz matrix, it is sufficient to prove that $f(T(a))-T(f\circ a)$ is a compact operator.

We have
\begin{multline}\label{eq:thm1}
  \|f(T(a))-T(f\circ a)-K_k\|_2 \\ \le
  \|f(T(a))-p_k(T(a))\|_2+\|T(f\circ a)-T(p_k\circ a)\|_2
   \le 2\|f-p_k\|_{\infty,\Sigma},
\end{multline}
where the last inequality follows by functional calculus, since $f-p_k$ is continuous in the spectrum $\Sigma$ of $T(a)$ and from the property  $\|T(a)\|_2=\|a\|_{\infty}$ (compare Lemma \ref{lem:toeplitz}).  
Whence we get
\[
	\|T(f\circ a-p_k\circ a)\|_2 =
	\|(f-p_k)\circ a\|_{\infty,[0,2\pi]}=\|f-p_k\|_{\infty,a([0,2\pi])}\le\|f-p_k\|_{\infty,\Sigma}.
\]
Inequality \eqref{eq:thm1} implies that $K_k$ converges to $f(T(a))-T(f\circ a)$ in the operator norm, and since $\mathcal K$ is closed, we may conclude that $f(T(a))-T(f\circ a)$ is a compact operator, that is what we wanted to prove.

Regarding the last statement, by Lemma \ref{thm:lemC1}, part (ii) and Lemma \ref{lem:toeplitz}, we know that $f(T(a))$ and $T(f\circ a)$ are self-adjoint and thus $K$ is self-adjoint as well, by Lemma \ref{thm:lemC1}, part (iii), and Lemma \ref{lem:toeplitz}, we get
\[
\mathrm{sp}(f(T(a)))=f(\mathrm{sp}(T(a)))=f(a([0,2\pi]))=\mathrm{sp}(T(f\circ a)).
\]
\end{proof}

By slightly modifying the above proof, we can easily arrive at the following generalization.

\begin{theorem}\label{thm:f(A)gen}
Let $a\in C_{2\pi}$ be a \rv\ function, $H$ a self-adjoint compact operator and $f$ a continuous function on the spectrum of $T(a)+H$, then
$f(T(a)+H)=T(f\circ a)+K$,
where $K$ is a compact operator. If $f$ takes real values on the spectrum of $T(a)$, then $K$ is self-adjoint.
\end{theorem}

An immediate consequence of the above result is the following corollary related to the weighted geometric mean.

\begin{corollary}\label{cor:0}
Let $A,B\in\QT$ be positive definite operators associated with the continuous symbols $a,b$, respectively. 
Then, $a,b>0$ and for $t\in[0,1]$, we have $G=A\#_t B\in\QT$ and the symbol $g$ associated with 
$G$ is such that $g=a^{1-t}b^t$.
\end{corollary}
\begin{proof} From Lemma \ref{lem:pos} we deduce that if $A$ is positive
definite then $a,b>0$.  
By definition, $A\#_t B=A^\frac{1}{2}(A^{-\frac{1}{2}}BA^{-\frac{1}{2}})^t A^{\frac{1}{2}}$ so that, by Theorem \ref{thm:f(A)gen} one has
$A^\frac{1}{2}, (A^{-\frac{1}{2}}BA^{-\frac{1}{2}})^t\in \QT$, thus $A\#_t B\in\QT$
and $g=a^{1-t}b^t$.
\end{proof}

The next lemma shows that if a sequence $\{A_k\}_k\subset\QT$ converges to $A\in\mathcal{B}(\ell^2)$, then $A\in\QT$ and the Toeplitz part and the compact part of $A_k$ converge to the Toeplitz part and the compact part of $A$, respectively.

\begin{lemma}\label{thm:lemma0}
Let $\{A_k\}_k$ be a sequence of quasi-Toeplitz matrices such that $A_k=T(a_k)+K_k\in\QT$, where $a_k\in C_{2\pi}$, and $K_k\in\K$, for $k\in\mathbb Z^+$. Let $A\in\mathcal B(\ell^2)$ be such that $\lim_k\|A-A_k\|_2=0$. Then $A\in\QT$, that is $A=T(a)+K$ with $a\in C_{2\pi}$ and $K\in\K$ and, moreover,  $\lim_k\|a-a_k\|_\infty=0$ and
$\lim_k\|K-K_k\|_2=0$.
\end{lemma}
\begin{proof}
Since $\QT$ is a $C^*$-algebra, then $\lim_k\|A-A_k\|_2=0$ implies that $A\in\QT$. 
From the inequality $\|T(a)\|_2\le \|A\|_2$ (see Lemma \ref{lem:toeplitz}, part (ii)) and from $\|a\|_\infty=\|T(a)\|_2$, it follows that $\|a-a_k\|_\infty=\|T(a-a_k)\|_2\le\|A-A_k\|_2$ tends to 0 and so does $\|K-K_k\|_2$.

\end{proof}

\subsection{Quasi-Toeplitz matrices with symbols in the Wiener algebra}

If $a(z)=\sum_{\ell=-\infty}^{+\infty} a_\ell z^\ell$ is a continuous function then  $\lim_{n}|a_n|=0$, therefore, for an error bound $\epsilon>0$ the cardinality of the set 
$\mathcal{S}_\epsilon=\{n\in\mathbb Z:~ |a_n|<\epsilon\}$ 
is finite. This fact allows one to numerically approximate the function $a(z)$ to any precision in a finite number of operations.

Classical results of harmonic analysis relate the regularity of $a(z)$ with the decay of the Fourier coefficients to zero. A better regularity implies a faster convergence to zero of $\{a_n\}_n$ and, in practice, this allows one to approximate $a(z)$ by using fewer coefficients.

We will consider Toeplitz matrices associated with functions having an analytic or an absolutely convergent Fourier series. The former situation is ideal since it implies an exponential decay to zero of $\{a_n\}_n$. 

We point out that if the function $a$ is just continuous then $\sum_{i=-\infty}^{+\infty}|a_i|$ is not necessarily finite so that $\|T(a)\|_\infty = \sum_{i=-\infty}^{+\infty}|a_i|$ (compare \cite[Theorem 1.14]{BG:book05}) is not bounded in general. 
This is a reason to determine conditions under which a function $f(a)$ for $a\in C_{2\pi}$, as well as the geometric mean of $a_i\in C_{2\pi}$, $i=1,\ldots,p$, have  absolutely summable coefficients or are analytic.

Another important related issue is to find out under which conditions a sequence of matrices $A_k\in \mathcal{QT}$ such that $\|A_k\|_\infty<\infty$  converges in the infinity norm to a limit $A\in\mathcal{QT}$ such that $\|A\|_\infty<\infty$.
In this section we provide tools to give an answer to these questions.

\bigskip

The set of functions $a\in L_{2\pi}^1$ with absolutely convergent Fourier series, namely
\[
	\W=\{f\in L_{2\pi}^1\,:\,f=\sum_{k\in\Z} a_k e^{\iunit k t},\,\sum_{k\in\Z}|a_k|<\infty\},
\]
with the norm $\|f\|_{\w}:=\sum_{k\in \Z} |a_k|$, is a Banach algebra, also known as the Wiener algebra \cite{BG:book00}.
A function $a\in\mathcal W$ is necessarily continuous, but it may fail to be differentiable, or even Lipschitz continuous.

We consider the convergence of a sequence in three broad classes of functions included in the Wiener algebra.

The first class is the set of $\alpha$-\Holder\ functions. A continuous function $\varphi:\Omega\to\C$, with $\Omega$ subset of $\R^n$ or $\C^n$, is said to be $\alpha$-\Holder\ on $\Omega$, with $0<\alpha \leq 1$, if
\[
	[\varphi]_\alpha := \sup_{\stackrel{x,y\in\Omega}{x\ne y}}\frac{|\varphi(x)-\varphi(y)|}{|x-y|^\alpha}
\]
is finite. We denote by $C_{2\pi}^{0,\alpha}$ the set of $\alpha$-\Holder\ functions on $\R$ with period $2\pi$, that is a Banach algebra with the norm
\[
\|\varphi\|_{C^{0,\alpha}}:=\|\varphi\|_\infty+[\varphi]_\alpha.
\]

The second class is the set of functions of bounded variation on $[0,2\pi]$, that is functions $f:[0,2	\pi]\to \C$, such that
\[
	V(f)=\sup_{\stackrel{N=1,2,\ldots}{x_0=0\le x_1\le \cdots\le x_N=2\pi}}\Bigl\{\sum_{i=0}^{N-1}|f(x_{i+1})-f(x_i)|\Bigr\}
\] 
is finite.

The last class is the set of absolutely continuous functions on $C_{2\pi}$, that is functions {$f:[0,2\pi]\rightarrow \mathbb{R}$ such that $f(2\pi)=f(0)$ and for any $\epsilon>0$, there exists $\delta>0$ such that 
\[
\sum_{i=1}^n|f(y_i)-f(x_i)|<\epsilon,
\]
whenever $\{[x_i,y_i]: i=1,\ldots, n\}$ is a finite collection of mutually disjoint subintervals of $[0,2\pi]$ with $\sum_{i=1}^n(y_i-x_i)<\delta$. 
}

\bigskip
Relations between $\W$ and these three classes are stated in the following summary of well-known results, that gives estimates of the norm $\|\cdot\|_{\w}$.

\begin{theorem}\label{thm:bernstein}
{Let $a\in C_{2\pi}$, we have }
\begin{enumerate}
\item[{\rm (i)}]
If $a\in C_{2\pi}^{0,\alpha}$, with $\alpha\in(1/2,1]$, then $a\in\mathcal W$. Moreover, there exists a constant $\gamma_a$ such that
\begin{equation}\label{eq:hoelder}
\|a\|_{\w}\le \gamma_a \|a\|_{C^{0,\alpha}}.
\end{equation}
\item[{\rm (ii)}]
If $a\in C_{2\pi}^{0,\alpha}$, with $\alpha\in(0,1]$ and of bounded variation $V(a)$ on $[0,2\pi]$, then $a\in\mathcal W$ and there exists a constant $\gamma_b$ such that
{
\begin{equation}\label{eq:holder+bv}
\|a\|_{\w}\le \gamma_b \bigl(|a_0|+V(a)^{1/2}\sum_{\ell=1}^\infty (\pi2^{-\ell})^{\alpha/2}\bigr).
\end{equation}
}
\item[{\rm (iii)}] If $a\in C_{2\pi}$ is absolutely continuous on $[0,2\pi]$ and $a'\in L_{2\pi}^2$, then $a\in\mathcal W$ and {there exists a constant $\gamma_c$ such that }
{
\begin{equation}\label{ab+l2}
\|a\|_{\w}\le \gamma_c\bigl(\|a\|_{L^1}+\|a'\|_{L^2}\bigr).
\end{equation}
}
\end{enumerate}
\end{theorem}
\begin{proof}
{ Concerning part (i), the} statement about $\alpha$-\Holder\ functions with $\alpha\in(1/2,1]$  is a classical result by Bernstein \cite{bernstein}, while the complete proof when $0<\alpha<1$ can {be found in \cite{Yitzhak} and the case where $\alpha=1$ can be proved analogously. }

{Concerning part (iii), the} statement about the absolute continuous function with derivative in $L_{2\pi}^2$ follows from  { \cite[Theorem I.6.2]{Yitzhak}}.

{ It is left to prove  part (ii).} If $a\in  C_{2\pi}^{0,\alpha}$ with $\alpha\in(0,1]$, and of bounded variation on $[0,2\pi]$, then by 
 \cite[Theorem I.6.4]{Yitzhak} (see also \cite[p.241]{zygmund}), it follows that $a\in\mathcal W$. We just need to prove the {inequality \eqref{eq:holder+bv}.}
	
{
	Denote by $\omega(\delta)=\omega(\delta,a)=\sup|a(t_1)-a(t_2)|$ for $t_1,t_2\in [0,2\pi]$ and $|t_1-t_2|\leq \delta$. 
	Applying the technique as in \cite[p.241-242]{zygmund}, we get
\begin{equation}\nonumber
\sum_{k=1}^{\infty}|a_k|\leq \frac{1}{2}\sum_{\ell=1}^{\infty}\bigl(\omega(\pi2^{-\ell})\bigr)^{1/2}V(a)^{1/2} \leq \frac{\sqrt{\gamma_1}}{2}\sum_{\ell=1}^{\infty}(\pi2^{-\ell})^{\alpha/2}V(a)^{1/2},
\end{equation}
 where the last inequality holds since $\omega(\delta)\leq r_1\delta^{\alpha}$ for a constant $r_1$ independent of $\delta$ if $a\in C^{0,\alpha}_{2\pi}$.  
It can be proved analogously that 
\begin{equation}\nonumber
\sum_{k=1}^{\infty}|a_{-k}|\leq \frac{\sqrt{\gamma_1}}{2}\sum_{\ell=1}^{\infty}(\pi2^{-\ell})^{\alpha/2}V(a)^{1/2}.
\end{equation}
 The proof is completed by choosing $\gamma_b=\max\{1,\sqrt{\gamma_1}\}$. 
}
\end{proof}

Note that, in particular, Theorem \ref{thm:bernstein} implies that a Lipschitz continuous function with period $2\pi$ belongs to $\W$.

\bigskip

We are interested in fractional powers and means of functions in the Wiener algebra. In the case of analytic functions there is an interesting result due to L\'evy \cite{levy}.
\begin{theorem}\label{thm:levy}
Let $a\in\W$ and let $f$ be a complex function, analytic in the range of $a$, then $f\circ a\in\W$.
\end{theorem}

Clearly, if both functions $a$ and $f$ are analytic then $f\circ a$ is analytic.
A simpler result that we will use in the following is related to \Holder\ functions.
\begin{lemma}\label{thm:lemmaB}
Let $a:\Omega\to \R$ be $\alpha$-\Holder\ and let $f\in C^1(\mathcal I)$ where $\mathcal I$ is a closed interval containing the range of $a$. Then $f\circ a$ is $\alpha$-\Holder\ and
\[
	[f\circ a]_\alpha \le
\|f'\|_{\infty,\mathcal I}
[a]_\alpha.
\]
Moreover $\|f\circ a\|_{C^{0,\alpha}}\le \gamma (\|f\|_{\infty,\mathcal I}+\|f'\|_{\infty,\mathcal I})$, where $\gamma=\max\{[a]_\alpha,1\}$.
\end{lemma}
\begin{proof}
If $x,y\in\Omega$ are such that $a(x)\ne a(y)$, then, by the mean value theorem,
\[
    \frac{|f(a(x))-f(a(y))|}{|x-y|^\alpha}=
    \frac{|f(a(x))-f(a(y))|}{|a(x)-a(y)|}
    \frac{|a(x)-a(y)|}{|x-y|^\alpha}\le \|f'\|_{\infty,\mathcal I}[a]_\alpha.
\]
The bound holds also when $a(x)=a(y)$ and we have $[f\circ a]_\alpha\le \|f'\|_{\infty,\mathcal I}[a]_\alpha$.

The latter inequality, follows from $\|f\circ a\|_{C^{0,\alpha}}=\|f\circ a\|_\infty+[f\circ a]_\alpha$.
\end{proof}

As a consequence we have.
\begin{corollary}\label{thm:h} If $a(t)\in\mathcal W$ is such that $a(t)>0$ then $a(t)^{1/p}\in\mathcal W$. If in addition $a(t)\in C_{2\pi}^{0,\alpha}$  for some $\alpha>0$ then $a^{1/p}\in C_{2\pi}^{0,\alpha}$ and $[a^{1/p}]_\alpha\le\frac{1}{p}\frac{[a]_\alpha}{\min a(t)^{1-1/p}}$.
\end{corollary}
\begin{proof}
Since $f(z)=z^{1/p}$ is analytic in the range of $a$, by Theorem \ref{thm:levy}, $a^{1/p}\in\W$. Let $\mathcal I=[\min a(t),\max a(t)]$ be a closed interval of positive numbers enclosing the range of $a$, then $f\in C^1(\mathcal I)$ and we can apply Lemma \ref{thm:lemmaB}, using $\|f'\|_{\infty,\mathcal I}=\frac{1}{p\min a(t)^{1-1/p}}$.
\end{proof}

Similarly, we have the following result.

\begin{lemma}\label{thm:lemmaB2}

{Let $a: [0, 2\pi]\rightarrow \mathbb R$ be a \rv\ function  and let $f\in C^1(\mathcal I)$ where $\mathcal I$ is a closed interval containing the range of $a$. 
\begin{enumerate}
\item[\rm (i)] If $a$ is absolutely continuous and $a'\in L_{2\pi}^2$, then  $f\circ a$ is absolutely continuous and with derivative in $L_{2\pi}^2$. Moreover, $\|(f\circ a)'\|_{L^2}\leq \|f'\|_{\infty,\mathcal I}\|a'\|_{L^2}$.

\item[\rm (ii)] If $a$ is of bounded variation on $[0,2\pi]$, then $f\circ a$ is of bounded variation on $[0,2\pi]$ and, moreover, $V(f\circ a)\le \|f'\|_{\infty,\mathcal I} V(a)$.
\end{enumerate}
}
\end{lemma}
\begin{proof}
 { Concerning part (i), it follows from \cite[Theorem 5.10,  part (d)]{Appell} that $f\circ a$ is absolutely continuous. Observe that $(f\circ a)'=(f'\circ a)a'$ and 
\[
\Bigl(\int_{0}^{2\pi}|(f\circ a)'|^2 dt\Bigr)^{1/2}\leq \Bigl(\int_{0}^{2\pi}\|f'\circ a\|_{\infty}^2|a'|^2 dt\Bigr)^{1/2}=\|f'\|_{\infty, \mathcal I}\|a'\|_{L^2}<\infty,
\]
which implies that $(f\circ a)' \in L_{2\pi}^2$ and $\|(f\circ a)'\|_{L^2}\leq \|f'\|_{\infty,\mathcal I}\|a'\|_{L^2}$.

}

{Concerning part (ii), a direct consequence of \cite[Theorem 5.10,  part (e)]{Appell} shows that $f\circ a$ is of bounded variation.}

Consider a partition $x_0=0\le x_1\le \ldots\le x_\ell=2\pi$, { we have} 
\[
 \begin{split}\label{bv1}
\sum_{j=0}^{\ell-1}|f(a(x_{j+1})) &-f(a(x_j))| =\sum_{\underset{a(x_j)\ne a(x_{j+1})}{j=0}}^{\ell-1}\frac{|f(a(x_{j+1}))-f(a(x_j))|}{|a(x_{j+1})-a(x_j)|} |a(x_{j+1})-a(x_j)|\\
    &\leq \sup_{\underset{x\ne y}{x,y\in \mathcal I}}\frac{|f(x)-f(y)|}{|x-y|} \sum_{j=0}^{\ell-1} |a(x_{j+1})-a(x_j)|\leq \|f'\|_{\infty,\mathcal I}V(a).
\end{split}
\]
Taking the supremum over all partitions, we get the desired inequality.
\end{proof}

\section{Geometric means of \QTmm}\label{sec:planar}

We start this section by  recalling  a  general construction of the ALM mean and the NBMP mean of $p$ $(p\geq 3)$ positive definite operators.  

Denote by $G_t(A,B)$, the weighted geometric mean $A\#_{t} B$, with weight $t\in [0,1]$.  Given the $(p-1)$-tuple $(s_1,s_2,\ldots, s_{p-1})$  with $s_i\in [0,1]$  and positive definite operators $A_1,\ldots, A_p$, the sequences  generated by
\begin{equation}\label{eq:almbmp}
A_i^{(k+1)}=A_i^{(k)}\#_{s_1}G_{s_2,\ldots,s_{p-1}}(A_1^{(k)},\ldots,A_{i-1}^{(k)},A_{i+1}^{(k)},\ldots,A_p^{(k)}), \ i=1,\ldots, p,
\end{equation}
with $A_i^{(0)}=A_i$, can be recursively defined, and they converge to a common limit $G_{s_1,\ldots,s_{p-1}}$ \cite{Hosoo}.

With the choice $(s_1,\ldots,s_{p-2},s_{p-1})=(1,\ldots,1,1/2)$ one obtains the ALM mean \cite{alm} and with the choice $(s_1,s_2,\ldots, s_{p-1})=((p-1)/p, (p-2)/(p-1), \ldots, 1/2)$ one obtains the NBMP mean. We call the corresponding iterations the ALM iteration and the NBMP iteration, respectively. The latter construction for positive definite matrices can be found  in \cite{bmp,nakamura}, while for positive definite operators  the convergence of the sequences in Thompson metric was proved in  \cite{nakamura}.

A similar inductive construction has been introduced in \cite{lll-12}. Given a probability vector 
$w=(w_i)\in\mathbb R^p$, i.e., such that $w_i>0$ and $\sum_{i=1}^p w_i=1$, define the following sequence
\begin{equation}\label{eq:weighted}
A_i^{(k+1)}=A_i^{(k)}\#_{1-w_i} G_{\hat w^{(i)}}(A_1^{(k)},\ldots,A_{i-1}^{(k)},A_{i+1}^{(k)},\ldots,A_p^{(k)}),\quad i=1,\ldots,p,
\end{equation}
where $\hat w^{(i)}=\frac1{1-w_i}(w_1,\ldots,w_{i-1},w_{i+1},\ldots,w_p)$ is again a probability vector, and $ G_{w_1,w_2}(A_1,A_2)=A_1\#_{w_2} A_2$.
If $\lim_k A_i^{(k)}$ exists and has the same value for every $i$, then we denote the common limit by 
$ G_{w}(A_1,\ldots,A_p)$ and refer to it as the weighted mean. It is proved in \cite{lll-12} that this limit exists in the Thompson metric. Observe that by choosing $w=\frac{1}{p}(1,1,\ldots,1)$, the weighted mean coincides with the NBMP mean.

In this section, we show that the ALM mean, the NBMP mean, the weighted mean and the Karcher mean of positive definite operators $A_1,\ldots, A_p$ such that  $A_i=T(a_i)+E_{i}\in \mathcal{QT}$ are also \QTmm. In fact, we prove that the sequences $\{A_i^{(k)}\}_k$ generated by \eqref{eq:almbmp} converge to a common limit $G\in \mathcal{QT}$ with symbol $g=(a_1\ldots a_p)^{\frac{1}{p}}$ and we have $\lim_k\|a_i^{(k)}-g\|_{\infty}=0$, where $a_i^{(k)}$ is the symbol of $A_i^{(k)}$.  
In the case where 
the symbols $a_i\in {\mathcal W}$, $i=1,\ldots,p$, we provide sufficient conditions under which  $a_i^{(k)}$ converges to $g$ in Wiener norm.
 
\subsection{ALM mean}\label{sec:alm}

The ALM sequences $\{A_i^{(k)}\}_{k=0}^{\infty}$ generated by \eqref{eq:almbmp} with 
$(s_1,\ldots,s_{p-2},s_{p-1})=(1,\ldots,1,1/2)$, converge to a common limit $G$ in the Thompson metric (see \cite[Remarks 4.2 and 6.5 and Theorem 4.3]{Hosoo}).
 This implies that $\lim_k\|A_i^{(k)}-G\|_2=0$ since the topology of the Thompson metric agrees with the relative operator norm topology \cite{thompson}.
 
We will prove that if the positive definite matrices $A_1,\ldots, A_p$, $p\ge 3$,  belong to $\mathcal{QT}$ then also the matrices $A_i^{(k)}$ of the ALM sequence generated by 
\eqref{eq:almbmp} as well as their limit $G$ belong to $\QT$.  Moreover, the symbol associated with the Toeplitz part of $G$ is the uniform limit of the symbols $a_i^{(k)}$ associated with the Toeplitz parts of $A_i^{(k)}$, which in turn are the functions obtained by applying the ALM construction to the symbols $a_i$ associated with the Toeplitz parts of the matrices $A_i$.

\begin{theorem}\label{thm:alm2norm}
Let $A_i=T(a_i)+E_i\in \mathcal{QT}$ be positive definite, for $i=1,\ldots, p$, with $p\ge 3$. The matrices $A_i^{(k)}$ generated by \eqref{eq:almbmp} for the ALM iteration, and the
ALM mean $G=G(A_1,\ldots,A_p)$ of $A_1,\ldots, A_p$, satisfy the following properties:
\begin{enumerate}
\item for any $k\ge 0$ and for $i=1,\ldots, p$, there exist $a_i^{(k)}\in C_{2\pi}$ and $K_i^{(k)}\in\mathcal K(\ell^2)$ such that $A_i^{(k)}=T(a_i^{(k)})+K_i^{(k)}$, that is $A_i^{(k)}\in\QT$;
\item there exist $g\in C_{2\pi}$ and $K_G\in\mathcal K(\ell^2)$ such that $G=T(g)+K_G$, that is, $G\in \mathcal{QT}$;
\item $\lim_k\|a_i^{(k)}-g\|_\infty=0$, $\lim_k\|K_i^{(k)}-K_G\|_2=0$, for $i=1,\ldots,p$;
\item the equation $a_i^{(k+1)} = a_i^{(k)}\#_{s_1}G_{s_2,\ldots,s_{p-1}}(a_1^{(k)},\ldots,a_{i-1}^{(k)},a_{i+1}^{(k)},\ldots,a_p^{(k)})$, is satisfied for $i=1,\ldots,p$, and for any $k$.\footnote{Here and in the proof of the theorem, we have $s_1=1$ and thus the notation could be simplified, but we prefer to keep $s_1$ to let the proof be used also to deal with the NBMP and weighted means.}
\end{enumerate}
\end{theorem}

\begin{proof}
 It is known from \cite{Hosoo} that  the sequences $\{A_i^{(k)}\}_{k=0}^{\infty}$ converge to $G$ in the Thompson metric and that $\lim_k\|A_i^{(k)}-G\|_2=0$.
We prove  parts 1--4 by induction on $p$.
If $p=3$, then 
\begin{equation}\label{eq:alm3}
\begin{aligned}
&A_1^{(k+1)}=A_1^{(k)}\#_{s_1}G(A_2^{(k)},A_3^{(k)}),\\
&A_2^{(k+1)}=A_2^{(k)}\#_{s_1}G(A_3^{(k)},A_1^{(k)}),\\
&A_3^{(k+1)}=A_3^{(k)}\#_{s_1}G(A_1^{(k)},A_2^{(k)}),\\
\end{aligned}
\end{equation}
Recall that if $A,B\in\QT$ then the geometric mean $G(A,B)\in\QT$, and the symbols $a,b,g$ associated with $A,B$ and $G(A,B)$, respectively, are such that $g=G(a,b)$ in view of Corollary \ref{cor:0}.

Using an induction argument on $k$, we show part 1 of the theorem i.e., $A_i^{(k)}\in\QT$ for $i\in\{1,2,3\}$. We have $A_i^{(0)}=A_i\in\QT$, and assuming $A_i^{(k)}\in\QT$, for $i\in\{1,2,3\}$, 
from \eqref{eq:alm3} and
Corollary \ref{cor:0}, we deduce that $A_i^{(k+1)}\in\QT$, for $i\in\{1,2,3\}$. Consequently, since $\lim_k\|A_i^{(k)}-G\|_2=0$, for $G=G(A_1,A_2,A_3)$,  then from Lemma \ref{thm:lemma0} we deduce part 2, i.e., $G\in\QT$ and that the symbol associated with the Toeplitz part of $A_i^{(k)}$
converges to $g$ uniformly and that the compact part of $A_i^{(k)}$ converges to the compact part of $G$ in norm, i.e., part 3.
Finally, since the symbol associated with $A_i^{(0)}$ is $a_i(z)$, we find that the symbol associated with $G(A_i^{(k)},A_j^{(k)})$ is $G(a_i^{(k)},a_j^{(k)})$ in view of Corollary \ref{cor:0}. Therefore, from \eqref{eq:alm3} and Corollary \ref{cor:0} we find that 
  { $a_i^{(k+1)} = a_i^{(k)}\#_{s_1}G(a_{i-1}^{(k)},a_{i+1}^{(k)})$ with $a_{0}^{(k)}:=a_3^{(k)}$ and $a_{4}^{(k)}:=a_{1}^{(k)}$.}  That is, part~4.
 
For the inductive step on $p$, assume $p>3$ and follow the same argument to prove that if 
parts 1--4 hold for the sequence generated by \eqref{eq:almbmp} starting from $p-1$ matrices $A_i\in\QT$, $i=1,\ldots,p-1$, then  they also hold for the sequence generated by \eqref{eq:almbmp} starting from $p$ matrices 
$A_i\in\QT$, $i=1,\ldots,p$.

To this end, consider equation \eqref{eq:almbmp} and use induction on $k$ to prove that $A_i^{(k)}\in\QT$ for $i=1,\ldots,p$. For $k=0$, clearly $A_i^{(0)}=A_i\in\QT$ by assumption.
  Concerning the inductive step on $k$, assume that $A_i^{(k)}\in\QT$ for $i=1,\ldots,p$, and deduce that $A_i^{(k+1)}\in\QT$ for $i=1,\ldots,p$.  By the inductive assumption on $k$ we have $A_i^{(k)}\in\QT$ so that by the inductive assumption on 
$p$, the matrix $G_{s_2,\ldots,s_{p-1}}(A_1^{(k)},\ldots,A_{i-1}^{(k)},A_{i+1}^{(k)},\ldots,A_p^{(k)})=G(A_1^{(k)}\ldots,A_{i-1}^{(k)},A_{i+1}^{(k)},\ldots,A_p^{(k)})$ belongs to $\QT$. Therefore, in view of Corollary \ref{cor:0} and \eqref{eq:almbmp} also $A_i^{(k+1)}$ is in $\QT$. That is part 1 of the theorem. Moreover, since $\lim_k\|A_i^{(k)}-G\|_2=0$ for $G=G(A_1,\ldots,A_p)$,  then from Lemma \ref{thm:lemma0} we deduce that $G\in\QT$ and that the symbol associated with the Toeplitz part of $A_i^{(k)}$ uniformly
converges to the symbol $g$ associated with the Toeplitz part of $G$, and that the compact part of $A_i^{(k)}$ converges to the compact part of $G$ in norm, i.e., parts 2 and 3 of the theorem.

Concerning part 4, we proceed by induction on $p$. We have already proved that for $p=3$ the property is satisfied. In order to prove the inductive step on $p$ we proceed by induction on $k$. For the initial step, i.e., for $k=0$, by Corollary \ref{cor:0} and by the inductive assumption on $p$, from \eqref{eq:alm3} we find that, $a_i^{(1)} = a_i^{(0)}\#_{s_1}G_{s_2,\ldots,s_{p-1}}(a_1^{(0)},\ldots,a_{i-1}^{(0)},a_{i+1}^{(0)},\ldots,a_p^{(0)})$. For the induction step on $k$, assume that part 4 is satisfied for $k$ and prove it for $k+1$. By the inductive hypotheses valid for $p-1$ matrices, we know that
\[
G_{s_2,\ldots,s_{p-1}}(A_1^{(k)},\ldots,A_{i-1}^{(k)},A_{i+1}^{(k)},\ldots,A_p^{(k)})\in\QT
\]
 and that its symbol is $G_{s_2,\ldots,s_{p-1}}(a_1^{(k)},\ldots,a_{i-1}^{(k)},a_{i+1}^{(k)},\ldots,a_p^{(k)})$. Therefore, by Corollary \ref{cor:0} and from \eqref{eq:almbmp}, the symbol associated with the Toeplitz part of  $A_i^{(k+1)}$ is $a_i^{(k+1)}= a_i^{(k)}\#_{s_1}G_{s_2,\ldots,s_{p-1}}(a_1^{(k)},\ldots,a_{i-1}^{(k)},a_{i+1}^{(k)},\ldots,a_p^{(k)})$.
\end{proof}

As a consequence of the above theorem we will show that the symbol $g$ associated with the Toeplitz part of $G$ is such that $g(z)=(a_1(z)\ldots a_p(z))^{\frac{1}{p}}$, where the symbols $a_i(z)$ take positive values in view of Lemma \ref{lem:pos} since $A_i$ are positive definite. In order to prove this representation of $g(z)$,
consider the sequences $a_i^{(k)}(z)$ defined by the ALM iteration, that is
\begin{equation}\label{eq:ak}
\begin{aligned}
&a_i^{(k+1)} = a_i^{(k)}\#_{s_1}G_{s_2,\ldots,s_{p-1}}(a_1^{(k)},\ldots,a_{i-1}^{(k)},a_{i+1}^{(k)},\ldots,a_p^{(k)}),\\
&a_i^{(0)}(z)=a_i(z),
\end{aligned}
\end{equation}
for $i=1,\ldots,p$, $k\ge 0$.
 It can be easily verified that
\begin{equation}\label{eq:alm}
a_i^{(k+1)}(z)=\left(\prod_{j=1,\, j\ne i}^p a_j^{(k)}(z)\right)^\frac{1}{p-1},\quad i=1,\ldots,p,\quad k=0,1,\ldots
\end{equation}
We have the following.

\begin{lemma}\label{th:11}
Let $a_1(z),\ldots,a_p(z)$ be continuous 
nonnegative 
functions and
let $a_i^{(k)}$, for $i=1,\ldots,p$ and $k=0,1,\ldots$, be the sequences defined by the ALM iteration \eqref{eq:ak}.
For $k=0,1,\ldots,$ we have
\[
a_i^{(k)}=a_i^{n_{k-1}}\prod_{j=1,\, j\ne i}^p a_j^{n_k},\qquad i=1,\ldots,p,
\]
where $n_k=\frac{1}{p}\bigl(1+\frac{(-1)^{k+1}}{( p-1)^k}\bigr)$. 
\end{lemma}
\begin{proof}
We proceed by induction on $k$. The case $k=0$, follows from $n_0=0$ and $n_{-1}=1$.
For $k>0$, from \eqref{eq:alm}, we obtain for the sequences $\{n_k\}_k$ the difference equation
$(p-1)n_{k+1}=(p-2)n_k+n_{k-1}$, with $n_0=0$ and $n_{-1}=1$, whose solution is
$n_k=\frac{1}{p}(1-(-(1/(p-1))^k))$.
\end{proof}

Observe that $\lim _k n_k
=\frac{1}{p}$ so that
$\lim_k a_i^{(k)}(z)=\left(\prod_{i=1}^p a_i(z)\right)^\frac{1}{p}$ pointwise as expected.
On the other hand, from Theorem \ref{thm:alm2norm} it follows that convergence is uniform.
We may conclude with the following.

\begin{theorem}
If $A_i=T(a_i)+K_i\in\QT$, for $i=1,\ldots,p,$  are positive definite operators, then the symbol $g(t)$ associated with $G=G(A_1,\ldots, A_p)$ is such that $g(t)=(a_1(t)\cdots a_p(t))^{\frac{1}{p}}$.
\end{theorem}

{Now, consider the case where $A_1, A_2, \ldots, A_p\in \QT$, are such that their associated symbols $a_i\in{\mathcal W}$ and $E_i\in \mathcal{K}(\ell^2)$ for $i=1,\ldots,p$.  It is clear that the sequences $\{A_i^{(k)}\}$ generated by the ALM iteration converge to a matrix $G\in \mathcal{QT}$ with symbol $g=(a_1\cdots a_p)^{\frac{1}{p}}$. Since ${\mathcal W}$ is a Banach algebra, 
we have  $g=(a_1\cdots a_p)^{\frac{1}{p}}\in {\mathcal W}$. Concerning the convergence of the symbols $\{a_i^{(k)}\}$ of the sequences $\{A_i^{(k)}\}$, we know that $\lim_k\|a_i^{(k)}-g\|_{\infty}=0$, but uniform convergence does not imply that $\lim_k\|a_i^{(k)}-g\|_{\w}=0$
 (see \cite[Page 34]{Yitzhak}).

{Now we show that under some regularity conditions, the sequences $\{a_i^{(k)}\}_k$ converge to $g$ in Wiener norm, i.e., $\lim_k\|a_i^{(k)}-g\|_{\w}=0$.}

\begin{theorem}\label{thm:unified}
Let $a_1,\ldots,a_p\in C_{2\pi}$ be { the symbols of the positive definite matrices $A_1,\ldots, A_p\in\QT$,}  and let $a_i^{(k)}$, for $i=1,\ldots,p$ and $k=0,1,2,\ldots,$ be the sequence obtained by the ALM iteration \eqref{eq:ak}. If {one of the conditions}
\begin{itemize}
\item[(a)] $a_i\in {C_{2\pi}^{0,\alpha}}$, with $\alpha\in(\frac{1}{2},1]$;
\item[(b)] $a_i\in {C_{2\pi}^{0,\alpha}}$, with {$\alpha\in(0,1]$} and of bounded variation;
\item[(c)] $a_i$ absolutely continuous and with derivative in $L_{2\pi}^2$;
\end{itemize}
for $i=1,\ldots,p$, is fulfilled, then $a_i^{(k)}\in\mathcal W$, $(a_1\cdots a_p)^{1/p}\in\mathcal W$ and $a_i^{(k)}\to_{\w}(a_1\cdots a_p)^{1/p}$.
\end{theorem}
\begin{proof}
Observe that, in all three cases, Theorem \ref{thm:bernstein} implies that $a_1,\ldots,a_p\in\mathcal W$, {so that $a_i^{(k)}, a_i^{1/p}$ and $(a_1\cdots a_p)^{1/p}$ belong to $\mathcal W$ in view of Theorem \ref{thm:levy}. To show $a_i^{(k)}\to_{\w}(a_1\cdots a_p)^{1/p}$, we can see from Lemmas \ref{thm:ba1} and \ref{th:11} that  it suffices to show $a_i^{n_k}\to_{\w}a_i^{1/p}$, where $n_k$ is defined in Lemma \ref{th:11} and  $\lim_k n_k=\frac{1}{p}$.}

With the notation of Lemma \ref{th:11}, we have $a_i^{n_k}-a_i^{1/p}=f_k\circ a_i$, where  
\[
f_k(x):=x^{n_k}-x^{1/p}=x^{n_k}(1-x^{(-1/(p-1))^k})
\]
is a sequence of analytic functions on $(0,\infty)$.  { Observe that $a_i$, for $i=1,\ldots,p$, is a strictly positive function in view of Lemma \ref{lem:pos}, let $\mathcal I$ be the set of the range of $a_i$, then $\mathcal I\subset (0,\infty)$ is a closed interval and the sequences $\{f_k\}_k$ and $\{f'_k\}_k$ converge to $0$ uniformly on $\mathcal I$ by
\cite[Theorem 1.2]{Lang}.
We show that $\|a_i^{n_k}-a_i^{\frac{1}{p}}\|_{\w}=\|f_k\circ a_i\|_{\w}\rightarrow 0$ under the three cases.} 

In case (a), we can use Lemma \ref{thm:lemmaB} and Theorem \ref{thm:bernstein} again to get the following
\[
	\|f_k\circ a_i\|_{\w}\le \gamma_a\|
f_k\circ a_i\|_{C^{0,\alpha}}\le \gamma_a\gamma (\|f_k\|_{\infty,\mathcal I}+\|f'_k\|_{\infty, \mathcal I})\to 0.
\]

In case (b),  by  {Lemmas 12 and 14}, {$f_k\circ a_i\in C_{2\pi}^{0,\alpha}$ and is of bounded variation with} $V(f_k\circ a_i)\le\|f'_k\|_{\infty,\mathcal I}V(a_i)\rightarrow 0$.

Set $(f_k\circ a_i)(t):=h_k(t)=\sum_{j\in \mathbb{Z}}h^{(k)}_je^{\iunit jt}$, then $\|h_k\|_{\infty}=\|f_k\|_{\infty,\mathcal I}\rightarrow 0$. Observe that $|h_0^{(k)}|\leq \frac{1}{2\pi}\int_{0}^{2\pi}|h_k(t)|dt\leq \|h_k\|_{\infty}$ so that $h_0^{(k)}\rightarrow 0$ as $k\rightarrow \infty$.

Together with {Theorem \ref{thm:bernstein} and \eqref{eq:holder+bv}}, it yields
\[
\|f_k\circ a_i\|_{\w}\leq \gamma_b\Big(|h_0^{(k)}|+V(f_k\circ a_i)^{\frac{1}{2}}\sum_{v=1}^{\infty}(\pi2^{-v})^{\frac{\alpha}{2}}\Big)\rightarrow 0.
\]

In case (c), we use { part (i) of Lemma \ref{thm:lemmaB2}  and \eqref{ab+l2}} to get
\[
\begin{split}
	\|f_k\circ a_i\|_{\w} & \le {
	\gamma_c(\|f_k\circ a_i\|_{L^1}+\|(f_k\circ a_i)'\|_{L^2})}\\
	& \le \gamma_c({\|f_k\|_{\infty, \mathcal I}+\|f'_k\|_{\infty, \mathcal I}\|a_i'\|_{L^2})\to 0.}
	\end{split}
\]
\end{proof}

\subsection{NBMP mean and weighted mean}\label{sec:bmp}
The analysis performed in the previous section can be repeated here concerning the NBMP mean.
In fact, since the $p$ sequences  $\{A_i^{(k)}\}_{k=1,2,\ldots}$, $i=1,\ldots,p$,  generated by \eqref{eq:almbmp} for the  NBMP iteration converge to a common limit $G$ in the Thompson metric \cite{Hosoo}, then they converge in the operator norm. Following an analysis similar to the one of Section \ref{sec:alm}, one can see that the NBMP mean of positive definite \QTmm\ is a \QTm. Moreover, since the NBMP construction applied to scalars converges in just one step, then we have that for the symbols $a_i^{(k)}$ obtained this way it holds that $a_i^{(k)}(z)=g(z)$ for $k\ge1$, $i=1,\ldots,p$. We may conclude with the following results which can be proved by adapting the proof of Theorem \ref{thm:alm2norm}.

\begin{theorem}\label{thm:bmp}
Let $A_i=T(a_i)+E_i\in \mathcal{QT}$, for $i=1,\ldots, p$, $p\ge 3$, be positive definite. Then the matrices $A_i^{(k)}$ generated by \eqref{eq:almbmp} for the NBMP iteration, and the
NBMP mean $G=G(A_1,\ldots,A_p)$ of $A_1,\ldots, A_p$, satisfy the following properties:
\begin{enumerate}
\item $A_i^{(k)}=T(g)+K_{i}^{(k)}\in\QT$, $i=1,\ldots,p$, for any $k\ge 1$, where $g=(a_1\cdots a_p)^{\frac{1}{p}}$;
\item  $G=T(g)+K_G\in \QT$;
\item $\lim_k\|K_{i}^{(k)}-{K_G}\|_2=0$, for $i=1,\ldots,p$.
\end{enumerate}
\end{theorem}

A similar argument can be used for the weighted mean: the $p$ sequences  $\{A_i^{(k)}\}_{k=1,2,\ldots}$, $i=1,\ldots,p$, generated by \eqref{eq:weighted} converge to their limit $G$ in the Thompson metric \cite{lll-12}, and thus then they converge in the operator norm, and as before, the weighted mean of positive definite \QTmm\ is a \QTm. The symbols $a_i^{(k)}$ obtained with this procedure are such that, for $k>1$, we have $a_i^{(k+1)}(z)=a_i^{(k)}\#_{1-w_i}G_{\hat w^{(i)}}(a_1^{(k)},\ldots,a_{i-1}^{(k)},a_{i+1}^{(k)},\ldots, a_p^{(k)})$ for $i=1,\ldots,p$. 
We can show that $a_i^{(k)}=a_1^{w_1}a_2^{w_2}\cdots a_p^{w_p}=G_w(a_1,\ldots,a_p)$ for $k\ge 1$ and get the following results which can be proved again by adjusting the proof of Theorem~\ref{thm:alm2norm}.

\begin{theorem}\label{thm:weighted}
Let $A_i=T(a_i)+E_i\in \mathcal{QT}$, for $i=1,\ldots, p$, $p\ge 3$, be positive definite. Let  $w=(w_i)\in\mathbb R^p$ be a probability vector, and 
 $A_i^{(k)}$ be the matrix sequences generated by \eqref{eq:weighted} for the weighted iteration. Finally, let
 $G=G_w(A_1,\ldots,A_p)$  be the weighted mean of $A_1,\ldots, A_p$.
 Then we have
\begin{enumerate}
\item $A_i^{(k)}=T(g)+K_{i}^{(k)}\in\QT$, $i=1,\ldots,p$, for any $k\ge 1$, where $g=a_1^{w_1}a_2^{w_2}\cdots a_p^{w_p}$;
\item  $G=T(g)+K_G\in \mathcal{QT}$;
\item $\lim_k\|K_{i}^{(k)}-K_G\|_2=0$, for $i=1,\ldots,p$.
\end{enumerate}
\end{theorem}

\subsection{Karcher mean}

Let $\mathcal B$ be a $C^*$-algebra, it has been proved in \cite{lawson-2020} that the  equation
\begin{equation}\label{eq:karcher}
\sum_{i=1}^p\log(X^{-1/2}A_iX^{-1/2})=0,
\end{equation}
where $A_1,\ldots,A_p\in\mathcal B$ has a strictly positive definite solution $X\in\mathcal B$, unique if $\mathcal B$ is the $C^*$-algebra of bounded operators over an Hilbert space.

This implies that the Karcher mean of quasi-Toeplitz matrices exists and it is unique.

\begin{theorem}
Let $T(a_i)+K_i\in\QT$, with $a_i>0$ for $i=1,\ldots,p$. There exists a unique solution $X$ of the equation \eqref{eq:karcher} and $X=T(g)+K_G\in\QT$ where $g=(a_1\cdots a_p)^{1/p}$.
\end{theorem}
\begin{proof}
Since $\QT$ is a $C^*$-subalgebra of $\mathcal B(\ell^2)$ then equation \eqref{eq:karcher} has a unique solution $X=T(g)+K_G$ in $\QT$.

Observe that by Theorem \ref{thm:f(A)}, and by the elementary arithmetic of quasi-Toeplitz matrices, we have that $X^{-1/2}=T(g^{-1/2})+K_1$, $X^{-1/2}A_iX^{-1/2}=T(a_i/g)+K_{2,i}$, $\log(X^{-1/2}A_iX^{-1/2})=T(\log(a_i/g))+K_{3,i}$ and finally
\[
	0=\sum_{i=1}^p \log(X^{-1/2}A_iX^{-1/2})=T(\log(a_1\cdots a_p/g^p))+K_4,
\]
with $K_1, K_{2,i},K_{3,i},K_4\in\mathcal K$, for $i=1,\ldots,p$.

By the uniqueness of decomposition of quasi-Toeplitz matrices, it follows that $g=(a_1\cdots a_p)^{1/p}$.
\end{proof}

\section{Computational issues}\label{sec:issues}
In this section we discuss some issues concerning the effective computation of the geometric mean of $A_1,\ldots, A_p\in\mathcal{QT}$. We rely on the CQT-Toolbox \cite{bmr} for computations in the $\QT$ algebra but in order to compute geometric means, we need to compute some
fundamental functions of \QTmm, namely, the $p$-th root and in particular the square root, that we will discuss in the following. 
{In this section, without loss of generality, we assume that the matrix
$A=T(a)+E_A \in \QT$ is such that  $0 \le a(z) \le 1$. This condition is satisfied in particular if $A$ is positive and 
$\|A\|_2\le 1$
since $\|a(z)\|_\infty=\|T(a)\|_2\le\|A\|_2$.}

\subsection{Square root}
The square root has been implemented in \cite{bmr}
in two different ways relying on the Denman and Beavers algorithm, and on the Cyclic Reduction (CR) algorithm, respectively. While the two algorithms are equivalent to the Newton method for a matrix $A$, if the initial value commutes with $A$, in finite arithmetic their behavior differs (see \cite[Section 6.3]{higham}). 
In our numerical tests, the CR algorithm provided better numerical results and thus it appears to be better suited for \QTmm. 

The CR algorithm for the square root is defined as follows

\begin{equation}\label{eq:cr1}
\begin{array}{ll}
Y_{k+1}=-Y_k W_k^{-1} Y_k, &  Y_0= I-A, \\[1ex]
W_{k+1}=W_k+2Y_{k+1}, & W_0=2(I+A),
\end{array}
\end{equation}
where we assume that all the matrices $W_k$ are invertible. In the finite dimensional case it follows that 
$\lim_k\frac14 W_k=A^\frac{1}{2}$, $\lim_kY_k=0$, where convergence holds in any 
operator norm.

The sequences obtained by the CR algorithm are related to the sequences obtained by the simplified Newton method
\begin{equation}\label{eq:sn}
	X_{k+1}=\frac{1}{2}(X_k+AX_k^{-1}),\qquad
	X_0=\frac{1}{2}(I+A).
\end{equation}
Indeed, $W_k=4X_k$ and $Y_k=2(X_k-X_{k-1})$, with $X_{-1}=A$ (see \cite{iannazzo-2003}).

We define the sequence $\{x_k(x)\}_k$ of rational functions of the variable $x$, by means of $x_{k+1}(x)=\frac{1}{2}(x_k(x)+xx_k^{-1}(x))$,
for $k=0,1,\ldots$, with $x_0(x)=\frac{1}{2}(1+x)$. Since the sequence $x_k(x)$ is obtained by applying the Newton method to the equation $z^2=x$, customary arguments show that for $x\in[0,1]$ the sequence $\{x_k(x)\}_{k=0,1,\ldots}$ is well-defined and monotonically decreasing to $\sqrt{x}$. In view of Dini's theorem we may conclude that convergence is uniform on $[0,1]$.

{From} the identity $X_k=x_k(A)$ we deduce the following.

\begin{corollary}\label{thm:cor22}
Let $A=T(a)+K\in\QT$ be self-adjoint and such that $a$ is \rv\ and $v^TAv\ge 0$ for any $v\in\ell^2$, $\|v\|=1$. Then the sequences $W_k$ and $Y_k$ generated by \eqref{eq:cr1} are such that $\lim_k\|Y_k\|_2=0$, $\lim_k\|W_k-4A^\frac12\|_2=0$.
\end{corollary} 
\begin{proof}
By scaling $A$, we can assume that $v^TAv\le 1$.
The spectrum ${\rm sp}(A)$ of $A$ is a compact set contained in $[0,1]$, in fact it is contained in the closure of the numerical range $\{v^TAv\,:\,v^Tv=1\}$ and this closure is contained in $[0,1]$. 
Let $\{X_k\}_k$ be the sequence obtained by \eqref{eq:sn} and $x_k(x)$ the corresponding scalar sequence. The function $\varphi_k(x):=x_k(x)-\sqrt{x}$ is continuous in $[0,1]$ and by Lemma \ref{thm:lemC1} we have
\[
	\|X_k-A^{1/2}\|_2=\|\varphi_k(X_k)\|_2=\|\varphi_k(x)\|_{\infty,\mathrm{sp}(A)}.
\]
Since $\mathrm{sp}(A)\subset[0,1]$, the latter tends to zero and we have that $\lim_k\|X_k-A^{1/2}\|_2=0$ and using $W_k=4X_k$ and $Y_k=2(X_k-X_{k-1})$ we obtain the proof. 
\end{proof}

{}From the above results 
it follows that the symbols associated with the matrices in \eqref{eq:cr1} uniformly converge to the symbols of their limit. Moreover the compact corrections associated with these matrix sequences converge in the $\mathcal B(\ell^2)$ norm to the compact corrections of their limit.

 A better convergence of the symbols can be obtained under the assumption of more regularity of $a(z)$ in view of Theorems \ref{thm:bernstein} and \ref{thm:levy}.

\subsection{$p$-th root}
It is well known that
if $A$ is an $n\times n$ Hermitian matrix, where $n$ is finite, with eigenvalues in $[0,1]$ then the sequence generated by Newton's iteration
$
Y_{k+1}=\frac{1}{p} Y_k((p-1)I+Y_k^{-p}A)$, $Y_0=I$,
is such that $\lim_k Y_k=A^{\frac{1}{p}}$ \cite{higham}. However, this iteration may encounter stability problems when implemented in floating point arithmetic. A stable version is based on the following iteration
\begin{equation}\label{eq:proot2}
\begin{array}{ll}
Y_{k+1}=Y_k(\frac{(p-1)I+M_k}{p}),&Y_0=I,\\[1ex]
M_{k+1}=(\frac{(p-1)I+M_k}{p})^{-p}M_k,&M_0=A.
\end{array}
\end{equation}
We prove that the iteration \eqref{eq:proot2} applied to a \QTm\ converges in norm. To this end we follow the same approach used for the square root. More precisely, we introduce the functional sequences
\begin{equation}\label{eq:prootfun}
\begin{array}{ll}
y_{k+1}(x)=y_k(x)(\frac{p-1+m_k(x)}{p}),&y_0(x)=1,\\[1ex]
m_{k+1}(x)=(\frac{p-1+m_k(x)}{p})^{-p}m_k(x),&m_0(x)=x.
\end{array}
\end{equation}
so that we have $Y_k=y_k(A)$, $M_k=m_k(A)$.
Now we prove the following result.

\begin{theorem} \label{thm:prootuniform}
For any $x\in[0,1]$, for the sequences \eqref{eq:prootfun} we have 
\begin{enumerate}
\item $0\le m_k(x)\le m_{k+1}(x)\le 1$, 
\item $x^\frac{1}{p}\le y_{k+1}(x)\le y_k(x)$,
\item $\lim_{k}\|m_k-1\|_{\infty}=0, \lim_k\|y_k-x^{\frac{1}{p}}\|_{\infty}=0$,
\end{enumerate}
where inequalities are strict if $x\ne 0,1$.
\end{theorem}
\begin{proof}
Concerning Part 1, we prove by induction that $0\le m_k(x)\le 1$. Clearly, since $m_0(x)=x$ we have  $0\le m_0(x)\le 1$. For the inductive step, we observe that the function $f(t)=t(\frac{p-1+t}p)^{-p}$ such that $m_{k+1} = f(m_k)$, maps monotonically the interval 
$[0,1]$ into itself, so that $0\le m_k\le 1$ implies $0\le m_{k+1}\le 1$. Similarly, for the inequality $m_k\le m_{k+1}$ we have $m_0\le m_1$ and the monotonicity of $f(t)$ inductively implies that $m_k\le m_{k+1}$.
Part 2, follows since $(\frac{p-1+m_k(x)}{p})<1$ under the assumption $0\le m_k(x)\le 1$, and we know that the limit of the sequence is $x^\frac 1p$.
Finally, Part 3 is proved by applying Dini's theorem.
\end{proof}

{}From the identities $Y_k=y_k(A)$ and $M_k=m_k(A)$ we deduce the following.

\begin{corollary}\label{thm:cor24}
Let $A=T(a)+E_A\in\QT$ be self-adjoint and such that $a$ is \rv\ and $v^TAv\ge 0$ for any $v\in\ell^2$, $v\ne 0$. Then the sequences $Y_k$ and $M_k$ generated by \eqref{eq:proot2} are such that $\lim_k\|M_k-I\|_2=0$, $\lim_k\|Y_k-A^\frac12\|_2=0$.
\end{corollary} 

{}As in the square root case, the symbols associated with the matrices in \eqref{eq:proot2} uniformly converge to the symbols of their limit and the compact corrections associated with these matrix sequences converge in the $\mathcal B(\ell^2)$ norm to the compact corrections of their limit. 

\begin{remark}\rm
Note that the proofs of Corollaries \ref{thm:cor22} and \ref{thm:cor24} rely uniquely on the fact that the spectrum of a self-adjoint operator is contained in the closure of the numerical range and thus they can be stated on the milder hypothesis that $A$ is a self-adjoint operator in $\mathcal B(\ell^2)$. 
\end{remark}

\subsection{Computing the symbol $g(z)$}

Observe that both the symbol of the ALM mean and the NBMP mean is $g(z)=(a_1(z)\ldots a_p(z))^{1/p}$. 
The application of the iterations of Section \ref{sec:planar} in the arithmetic of quasi-Toeplitz matrices, provides as a result both the values of $g(z)$ and of the correction $K$ such that $G=T(g)+K$ is the sought geometric mean. However, if only the symbol part of $G$ is needed, then it might be more convenient to compute the coefficients of $g(z)$ by means of the evaluation/interpolation technique. 

In this section,  
{when the symbols are such that $a_{i}\in \mathcal W$, $i=1,\ldots,p$, } based on the evaluation/interpolation at the roots of unity, we provide  an algorithm for computing the approximation $\wt{g}_{j}$, $j=-n+1,\ldots,n$, to the Fourier coefficients $g_j$ of the symbol $g(z)=(a_1(z)\cdots a_p(z))^{1/p}$.

Let $n>0$ be an integer, set $m=2n$ and let $w_m=\cos\frac{2\pi}{m}+\iunit\sin\frac{2\pi}{m}$ be the principal $m$-th root of 1. There is always a unique Laurent polynomial  $\wt{g}(z)=\sum_{-n+1}^n\wt{g}_jz^j$ such that $g(w_m^{\ell})=\wt{g}(w_m^{\ell})$, that is, $\wt{g}(z)$ interpolates $g(z)$ at $w_m^{\ell}$, $\ell=-n+1,\ldots,n$.

The computation of the approximation $\wt{g}_j$, $j=-n+1,\ldots,n$, to the coefficients $g_j$ of $g(z)$ can proceed by first selecting a positive integer $n$ and evaluating $g(w_m^{\ell})$, $\ell=-n+1,\ldots,n$; and then interpolating $g(z)$ at $w_m^{\ell}$, $\ell=-n+1,\ldots,n$, by means of the FFT, obtaining the coefficients $\wt{g}_j$ of $\wt{g}(z)=\sum_{j=-n+1}^n\wt{g}_jz^j$. We stop the process if $\wt{g}(z)$ is close enough to $g(z)$, otherwise we continue this process by doubling the value of $n$.

Concerning the accuracy of the approximation, we recall the following result from \cite[Theorem 3.8]{blm:book}
\begin{equation}\label{eq:fftbound}
\wt g_i = g_i +\sum_{k=1}^\infty (g_{i+kn}+g_{i-kn}).
\end{equation}
If
$g\in\mathcal W$ then $\lim_n\sum_{|j|\ge n} |g_j|=0$ so that, for a given $\epsilon>0$ there exists a sufficiently large $n$ such that 
$\sum_{|j|\ge n} |g_j|\le \epsilon$. Thus, from \eqref{eq:fftbound}  we deduce that $|\wt g_i - g_i|\le\epsilon$.
Under additional assumptions on $g(t)$, a guaranteed way for determining $n$ such that the above bound is satisfied, can be easily determined (see \cite[page 57]{blm:book} for further details). In general, we may adopt the following heuristic criterion, to halt the evaluation/interpolation procedure (see also \cite{robol}) where the iteration is terminated if

\begin{equation}\label{stopping}
\sum_{|j|>\lceil\frac{n}{2}\rceil}|\wt{g}_{j}|< \sum_{j=-n+1}^n|\wt{g}_{j}|\cdot \epsilon.
\end{equation}

We summarize the procedure for computing the coefficients 
$\wt{g}_j$, $j=-n+1,\ldots,n$, as Algorithm \ref{alg:g}. 
The overall computational cost of this algorithm is $O(n \log n)$ 
arithmetic operations.

\begin{algorithm}[h]
\caption{Approximation of $g(z)$}
\label{alg:g}
 \begin{algorithmic}[1]
 \REQUIRE{The coefficients of $a_i(z)$, $i=1,\ldots,p$,  and a tolerance $\epsilon>0$.}
\ENSURE{Approximations $\wt g_j$, $j=-n+1,\ldots,n$,   to the coefficients $g_j$ of $g(z)$
 such that $\sum_{|j|>\lceil\frac{n}{2}\rceil}|\wt{g}_{j}|< \sum_{j=-n+1}^n|\wt{g}_{j}|\cdot \epsilon$.}

\STATE{Set $n=4$, $m=2n$ and $w_m=\cos\frac{2\pi}{m}+\iunit \sin\frac{2\pi}{m}$. Evaluate $a_i(z)$ at $z=w_m^{\ell}$ for $\ell=-n+1,\ldots, n$ and for $i=1,\ldots,p$;
}
\STATE{For $\ell=-n+1,\ldots, n$, compute the $p$-th root $r_{\ell}$ of $(a_1(z)\cdots a_p(z))$ at $z=w_m^{\ell}$;
}
\STATE{Interpolate the values $r_{\ell}$, $\ell=-n+1,\ldots, n$, by means of FFT and obtain the coefficients $\wt{g}_i$ of the Laurent polynomial $\wt{g}(z)=\sum_{j=-n+1}^n\wt{g}_jz^j$ such that $g(w_m^{\ell})=\wt{g}(w_m^{\ell})$, $\ell=-n+1,\ldots,n$;
}
\STATE{Compute $\delta_m=\sum_{|j|>\lceil  \frac{n}{2}\rceil}|\wt{g}_{j}|$ and $\kappa_m=|\wt{g}_0|+\sum_{j=-n+1}^n|\wt{g}_{j}|$;
}
\STATE{If $\delta_m < \kappa_m\epsilon$ then exit, else set $n=2n$ and compute from Step 2.
}
\end{algorithmic}
\end{algorithm}

Observe that if $a_i>0$ for $i=1,\ldots,p$, has only real coefficients,  then $a_i(w_m^\ell)=a_i(w_m^{-\ell})$ for $\ell=-n+1, \ldots, n$ and $i=1,\ldots,p$, then  Step 2 of Algorithm \ref{alg:g} can proceed by computing $a_i(w_m^{\ell})$ for $\ell=0,\ldots,n$ and setting $r_{\ell}=r_{-\ell}$ for $\ell=-1,\ldots, -n+1$.

\section{Geometric means of Finite \QTmm} \label{sec:finite}

The representation and the arithmetic for \QTmm, up to a certain extent, can be adapted for handing finite dimensional matrices. Accordingly, we show that {iteration \eqref{eq:almbmp}} for computing the mean $G$ can be applied to finite dimensional matrices that can be written as a sum of a Toeplitz matrix and a low-rank matrix correction.

Given a symbol $a(z)$ and $m\in \mathbb{Z}^+$, we denote by $T_m(a)$, $H_m(a^{-})$ and $H_m(a^+)$ the $m\times m$ leading principal submatrices of $T(a)$, $H(a^-)$ and $H(a^+)$, respectively. The following theorem can be seen a finite dimensional version of Theorem \ref{th:widom}, which implies that our algorithms for geometric means of \QTmm\ can be applied also for finite size matrices.

\begin{theorem}[\cite{Widom}]\label{thm:finite}
If $a,b\in L^{\infty}$, then
\[
T_m(a)T_m(b)=T_m(ab)-H_m(a^-)H_m(b^+)-JH_m(a^+)H_m(b^-)J,
\]
where $J$ is the $m\times m$ flip matrix having 1 on the anti-diagonal and zeros elsewhere.
\end{theorem}

If we focus on finite size Toeplitz matrices, whose symbols are Laurent polynomials of the kind $\sum_{j=-r}^r a_jz^j$, where the degree $r$ is small compared to the matrix size $m$, say, $r\le m/4$, then Theorem \ref{thm:finite} shows that the product of matrices of this kind can be represented as the sum of a finite Toeplitz matrix and two low-rank matrix corrections with nonzero entries (the support of the correction) located in the upper leftmost and in the lower rightmost corners, respectively. Observe also that if the matrices are real symmetric or complex Hermitian, then $a^-=a^+$ and $b^-=b^+$ so that the compact correction obtained in the infinite case provides both the corrections in the upper leftmost and the lower rightmost corner. 

A similar property holds 
if  the matrices can be written as the sum of a band Toeplitz matrix associated with a Laurent polynomial of degree at most $r$, and a correction with support located in the two opposite corners, provided that the values of $r$ and of the size $s$ of the support are suitably smaller than $m$, say, $r,s\le m/4$.
Moreover, comparing Theorem \ref{thm:finite} with Theorem \ref{th:widom} one can see that the Toeplitz part $T_m(ab)$ and the correction $H_m(a^-)H_m(b^+)$ coincide with the $m\times m$ leading principal submatrices of the infinite matrices $T(ab)$ and $H(a^-)H(b^+)$, respectively, so that the infinite $\QT$ arithmetic provides the corresponding result of the finite $\QT$   arithmetic. This property holds even in the case where there is an overlapping of the supports in the two corners of the corrections, or if the bandwidth $r$ takes large values.
In this situation, the implementation of the arithmetic is still possible \cite{bmr}, but its not efficient from the computational point of view.

This allows us to implement the computation of the sequences $A_i^{(k)}$ of $m\times m$ matrices generated by \eqref{eq:almbmp} relying on the computation of their infinite counterparts. This implementation leads to an effective computation if the numerical degree of the symbols as well as the size (the maximum between rows and columns) of the supports of the corrections of the matrices $A_i^{(k)}$ remain bounded from above by a constant smaller than or equal to $m/4$.

\section{Numerical experiments}\label{sec:exp}

In order to apply the theoretical results of the previous sections, we show by some tests the effectiveness of computing the ALM and the NBMP means of three matrices $A_1,A_2,A_3\in\QT$, and the convenience of computing the associated symbol $g$ with the evaluation/interpolation method. 

The algorithms of Section \ref{sec:issues} have been implemented in MATLAB by following the lines of the corresponding implementations for finite positive definite matrices of the Matrix Means Toolbox (\url{http://bezout.dm.unipi.it/software/mmtoolbox/}). They rely on the package CQT-Toolbox of \cite{bmr} for the storage and arithmetic of \QTmm.  The software can be provided by the authors upon request. 

The tests have been run on a cluster with 128GB of RAM and 24 cores. The internal precision of the CQT toolbox has been set to {\tt threshold = 1.e-15}, while the value of $\epsilon=\tt 10 * threshold=1.e-14$ has been used to truncate the values of the symbol and of the correction. That is, only the values $g_i$, $i=0,\ldots,k$ are computed, where $k$ is such that $|g_j|\le\epsilon\|g\|_\infty$ for $j>k$. Similarly, for the correction $E_G$, written in the form $E_G=UV^T$, where $U$ and $V$ are thin matrices, we computed only the values $u_{i,j}$, $i=1,\ldots,m$  and $v_{i,j}$, $i=1,\ldots,n$, where $m,n$ are such that $|u_{i,j}|<\epsilon\max_i|u_{i,j}|$ for any $i>m$ and $j$ and $|v_{i,j}|<\epsilon\max_i|v_{i,j}|$ for any $i>n$ and $j$, respectively. 

The test examples have been constructed by relying on trigonometric symbols in the class $f(t)=f_0+2f_1\cos(t)+2f_2\cos(2t)$, for different values of the coefficients $f_0,f_1,f_2$. A Toeplitz matrix associated with a symbol of this type turns out to be pentadiagonal.

With the choices of $(f_0,f_1,f_2)$ in the set $\{(2,1,0),~ (3,2,1),~ (9,4,4)\}$, the symbol $f(t)$ takes values in the intervals $[0,4]$, $[0,9]$, $[0,25]$, respectively, and 0 belongs to the image of the symbol in all three cases. 
By perturbing the value of $f_0$ into 
$f_0+\theta$ where $\theta>0$, we are able to tune the ratio $\max_t f(t)/\min_t f(t)$ which for finite matrices is related to the condition number of the associated Toeplitz matrix.
More specifically, with the choices $\theta=1,0.1,0.01$ we have three sets of matrices with increasing condition numbers.

In Table \ref{tab:1}, for each value of $\theta$ we report the numerical length of the symbol $g(z)$ together with the CPU time needed by the evaluation/interpolation technique to compute the coefficients of $g(z)$. The CPU time needed for this computation is quite negligible and the numerical length of the symbol, as well as the number of interpolation points, grow as the ratio $\min_t g(t)/\max_t g(t)$ gets closer to 0, as expected. 

\begin{table}
\begin{center}
\begin{tabular}{c|ccc}
$\theta$&length&$n$&CPU\\\hline
1&110&512&{\tt 1.7e-4}\\
0.1&317&2048&{\tt 4.2e-4}\\
0.01&926&4096&{\tt 6.3e-4}
\end{tabular}
\caption{Length of the symbol, number of interpolation points and CPU time in seconds for computing the coefficients of $g(z)$.}\label{tab:1}
\end{center}
\end{table}

In Table \ref{tab:2}, for each value of $\theta$ we report the number of iterations and the CPU time needed to compute the ALM mean $G_{\scriptsize\rm ALM}$ and the NBMP mean $G_{\scriptsize\rm NBMP}$, together with  the numerical size and the numerical rank of the compact correction. 

We may observe that the NBMP iteration arrives at numerical convergence in just 3 steps, while the ALM iteration requires 43 steps independently of the condition number of the matrices. 
This fact reflects the different aysmptotic convergence order of the two iterations for finite size matrices: while the ALM converges linearly, the NBMP converges cubically \cite{bmp}.

As for the symbol, the size and the rank of the correction grows when the matrices to be averaged get more ill conditioned. For this computation, the CPU time needed is much larger than the time needed to compute just the symbol $g(t)$.
A closer analysis, performed by using the MATLAB profiler, shows that the most part of time is spent to perform compression operations in the CQT-Toolbox.
By compressing a matrix $E$ with respect to a threshold value $\varepsilon$, we mean to find a matrix $\widetilde E$ of the lowest rank such that $\|E-\widetilde E\| \le \varepsilon$.

 \begin{table}
\begin{center}
\begin{tabular}{c|cccc|cccc}
\multicolumn{1}{c}{}&\multicolumn{4}{c}{ALM} & \multicolumn{4}{c}{NBMP}\\\hline
$\theta$&iter.&CPU&size&rank  &iter.&CPU&size&rank     \\\hline
1&43&{\tt 2.5e1}& 95&17   &3&{\tt 4.5e0}&108&16 \\
0.1&43&{\tt 1.2e2}&290&30  &3&3.2e1&345&28 \\
0.01&43&{\tt 1.8e3}&1220&44     &3&9.4e2&2506&46
\end{tabular}
\caption{Number of iterations, CPU time in seconds, size and rank of the correction in the computation of the ALM and the NBMP mean.}\label{tab:2}
\end{center}
\end{table}

In order to give an idea of the structure of the geometric mean, for $\theta=1$, in Figure \ref{fig:1} we show in logarithmic scale the graph of the symbol $g(t)$, i.e., the plot of the pairs $(j,\log_{10}|g_j|)$ for $j\ge 0$, and of the compact corrections $E=(e_{i,j})$, i.e., the plot of the triples $(i,j,\log_{10}|e_{i,j}|)$ for $G_{\scriptsize\rm ALM}$. While in Figure \ref{fig:2} we show the graph of the correction of
$G_{\rm\scriptsize NBMP}$ together with the modulus of componentwise difference $G_{\rm\scriptsize ALM}-G_{\rm\scriptsize NBMP}$. 

We may see that the geometric mean, even though is an infinite dimensional matrix, can be effectively approximated by a finite number of parameters once it is decomposed into its Toeplitz part and its compact correction. From the graphs in Figures \ref{fig:1} and \ref{fig:2} we may appreciate some rounding errors in the compact corrections. For the ALM mean, errors are located in the components of modulus less than $10^{-20}$, while for the NBMP mean the affected components have modulus less than $10^{-18}$ along an edge of the domain.

Since the ALM and the NBMP means differ only in the compact correction, an interesting remark is that this difference has values of modulus less than $10^{-5}$, see Figure \ref{fig:2}. This reflects also in the infinite dimensional case, the relatively small difference that has been usually observed between the two mean in the finite dimensional case \cite{bini-iannazzo-2011}. 

\pgfdeclareimage[width=5cm]{symbol}{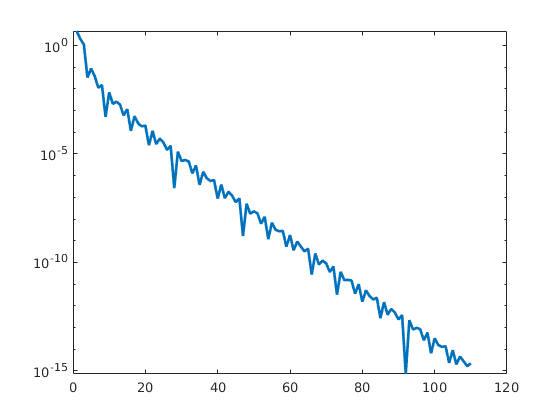} 
\pgfdeclareimage[width=5.5cm]{correctiondiff}{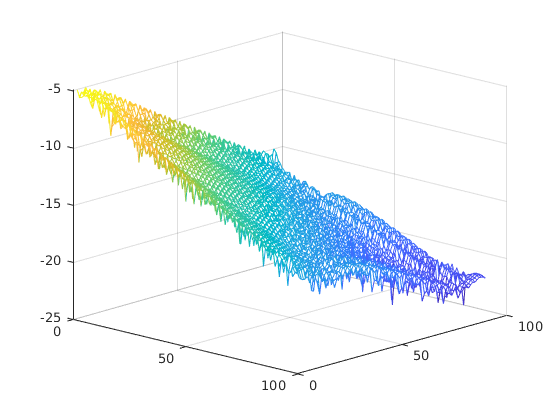}
\pgfdeclareimage[width=5.5cm]{correctionalm}{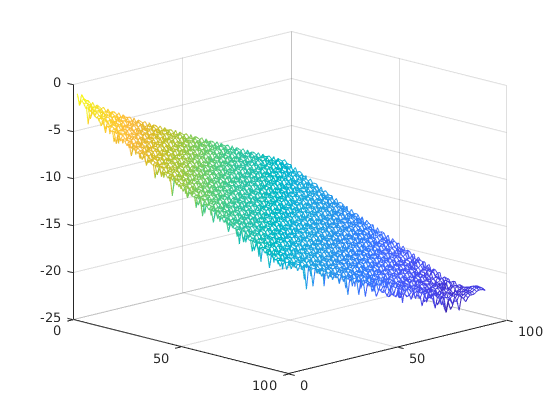}
\pgfdeclareimage[width=5.5cm]{correctionbmp}{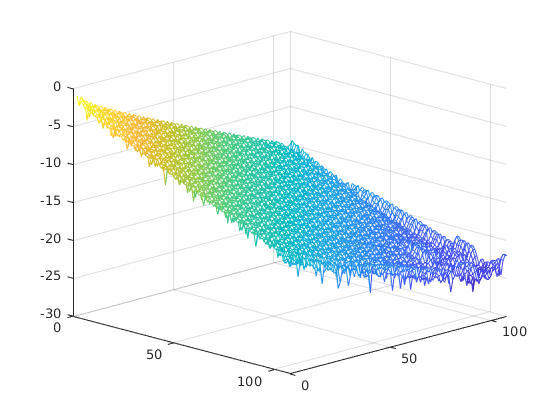}
\begin{figure}
\pgfuseimage{symbol} \pgfuseimage{correctionalm}
\caption{Absolute value of the symbol and of the correction of the ALM mean in log scale.}\label{fig:1}
\end{figure}

\begin{figure}
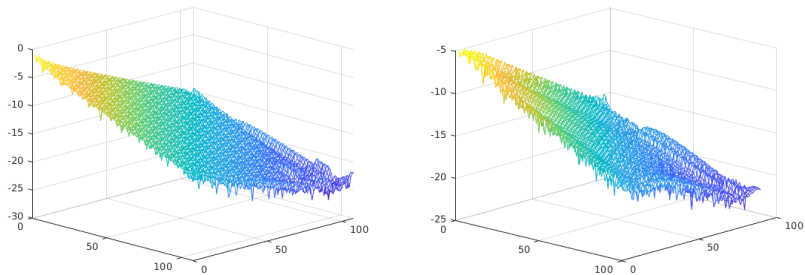

\pgfuseimage{correctionbmp} \pgfuseimage{correctiondiff}
\caption{Absolute value of the symbol of the NBMP mean (left) and of the difference between the ALM and the NBMP mean (right) in log scale.}\label{fig:2}
\end{figure}

 \begin{table}
\begin{center}
\begin{tabular}{c|cc}
$\theta$& ALM&NBMP\\\hline
1   &\tt 1.3e1 [2.5e1] &\tt 2.5e0 [4.5e0]\\
0.1 &\tt 1.2e2 [1.2e2] &\tt 3.1e1 [3.2e1]  \\
0.01&\tt 5.5e3 [1.8e3] &\tt 9.1e3 [9.4e2]
\end{tabular}
\caption{CPU time in seconds, needed to compute the ALM and the NBMP means in the case of finite matrices whose size is 3 times the size of the corresponding compact correction. For comparison, the CPU time needed in the infinite case is written between bracket.}\label{tab:3}
\end{center}
\end{table}

Finally, we have compared the CPU time needed by the ALM and NBMP iterations applied to \QTmm\ and to their finite truncation to size $n$. The value of $n$ is chosen equal to 3 times the size of the compact correction. This choice is motivated by the goal to keep well separated the compact correction in the top leftmost corner from that in the bottom rightmost corner of the finite size matrix. The CPU time is reported in Table \ref{tab:3}.
Observe that, while in the case $\theta=1$ of well conditioned matrices, the application of the ALM and NBMP iterations is faster for finite size matrices than for \QTmm, in the slightly more ill conditioned cases the time needed for finite size matrices gets larger than the time taken for \QTmm. In particular, for $\theta=0.01$ the speed-up reached by the computation based on the quasi-Toeplitz technology is roughly 3 for the ALM mean and 9.7 for the NBMP mean.

\section{Conclusions and open issues}\label{sec:conc}
The common definitions of geometric means of positive definite matrices, namely the ALM, NBMP, weighted and Karcher mean, have been extended to the case of infinite matrices which are the sum of a Toeplitz matrix, associated with a continuous symbol,  plus a compact correction. We have shown that these means still keep the same structure  of the input matrices, i.e., they can be written as the sum of a Toeplitz matrix and a compact correction. Moreover, the symbol associated with the Toeplitz part of the mean is the mean of the symbols associated with the averaged matrices. 
We have given, moreover, conditions under which the geometric mean belongs to $\mathcal B(\ell^\infty)$ and under which the ALM sequence of the symbols converges in Wiener norm.

Numerical computations show that the ALM and the NBMP means can be computed also for infinite matrices in the $\QT$ class and that these means can be represented with good precision in terms of a finite number of parameters. Perhaps surprising, these ideas can be useful also when computing means of finite matrices.

An open issue, which we would like to investigate further, is exploiting the availability of the symbol $g(t)$ of the geometric mean $G$ in order to accelerate the computation of the compact correction. 
In fact, even though $g(t)$ can be computed separately at a much lower cost, our current implementations of the ALM and NBMP iterations do not take advantage of the availability of $g(t)$ and compute simultaneously approximation to the symbol and to the compact correction of the mean. Moreover, the large CPU time needed for computing the correction is mainly spent for the operation of low-rank approximations of large matrices. We believe that this part can be much improved by means of randomized techniques both for the task of computing general matrix functions and for computing the geometric mean. 

\bibliographystyle{abbrv}
\bibliography{references}

\end{document}